\documentclass{IEEEtran}
\usepackage[utf8]{inputenc}
\usepackage{amsmath, amssymb, amsthm, mathrsfs}

\usepackage{graphicx}
\usepackage{todonotes}
\usepackage{xcolor}
\usepackage{algorithm}
\usepackage{algpseudocode}
\usepackage{lipsum}
\usepackage{tikz}
\usepackage{multirow}
\usetikzlibrary{shapes,arrows,intersections, fit}
\usepackage{adjustbox}

\algdef{SE}[DOWHILE]{Do}{doWhile}{\algorithmicdo}[1]{\algorithmicwhile\ #1}%

\newtheorem{thm}{Theorem}

\newtheorem{lem}{Lemma}

\newtheorem{prob}{Problem}

\usepackage{xspace}
\usepackage{amsmath,amssymb,amsfonts}

\newcommand{\bs}[1]{\ensuremath{\boldsymbol{#1}}}

\newcommand{\bv}{\ensuremath{\bs v}\xspace}
\newcommand{\bw}{\ensuremath{\bs w}\xspace}
\newcommand{\bx}{\ensuremath{\bs x}\xspace}

\newcommand{\bW}{\ensuremath{\bs W}\xspace}
\newcommand{\bX}{\ensuremath{\bs X}\xspace}

\newcommand{\bZ}{\ensuremath{\bs Z}\xspace}

\newcommand{\pci}[1]{\ensuremath{\bs{p}_{\boldsymbol{c}_{i}}}\xspace}

\newcommand{\PhiGaussian}{\Phi_\text{StdNorm}}
\newcommand{\UPWAm}{m^-}
\newcommand{\UPWAc}{c^-}

\newcommand{\bXU}{\bX_{\overline{U}}}
\newcommand{\Xd}{\overline{X}_{\mathrm{d}}}
\newcommand{\mubW}{\overline{\mu}_{\bW}}
\newcommand{\CbW}{C_{\bW}}
\newcommand{\mubXU}{\overline{\mu}_{\bX, \overline{U}}}
\newcommand{\CbXU}{C_{\bX, \overline{U}}}

\newcommand{\ProbbXU}{\mathbb{P}_{\bX}^{\overline{U}}}

\newcommand{\deltalb}{\delta_\mathrm{lb}}
\usepackage{array}
\newcolumntype{M}[1]{>{\centering\arraybackslash}m{#1}}
 
\newcommand{\Exp}{\mathbb{E}}
\title{Convexified Open-Loop Stochastic Optimal Control for Linear Non-Gaussian Systems}
\author{Vignesh Sivaramakrishnan$^\ast$, Abraham P. Vinod$^\ast$, and Meeko M. K. Oishi
 \thanks{This material is based upon work supported by the National Science Foundation under NSF Grant Number CNS-1836900, and by the Air Force Office of Scientific Research under  AFRL Grant No. FA9453-17-C-0087.  Any opinions, findings, and conclusions or recommendations expressed in this material are those of the authors and do not necessarily reflect the views of the National Science Foundation. \newline
         \indent Vignesh Sivaramakrishnan, and Meeko Oishi are with the Electrical \& Computer Engineering, University of New Mexico, Albuquerque, NM, US. Email: {\tt vigsiv@unm.edu; oishi@unm.edu}.\newline
         \indent Abraham Vinod is with Mitsubishi Electric Research Laboratories (MERL), Cambridge, MA, 02139, USA. Email: {\tt aby.vinod@gmail.com} This work was completed, while Vinod was a post-doctoral research fellow at the
University of Austin, Texas.\newline
         $^\ast$ These authors contributed equally to this work.\newline
         }
}
\date{}

\begin{document}
\maketitle

\begin{abstract}
    We consider stochastic optimal control of linear dynamical systems with additive non-Gaussian disturbance. 
    We propose a novel, sampling-free approach, based on Fourier transformations and convex optimization, to cast the stochastic optimal control problem as a difference-of-convex program.
    In contrast to existing moment based approaches, our approach invokes higher moments, resulting in less conservatism.
    We employ piecewise affine approximations and the well-known convex-concave procedure, to efficiently solve the resulting optimization problem via standard conic solvers.
    We demonstrate that the proposed approach is computationally faster than  existing particle based and moment based approaches, without compromising probabilistic safety constraints.
\end{abstract}

\section{Introduction}

Stochastic optimal control requires enforcement of chance
constraints, which permit violation of the state constraints
with a probability below a specified threshold \cite{HomChaudhuriACC2017,vitus2016stochastic,LesserCDC2013,blackmore2011chance}. Chance constraints trade off constraint violation with the objective cost.  
However, such constraints are hard to implement in a computationally tractable manner, especially for systems with non-Gaussian disturbances.
In this paper, we propose a method for stochastic optimal control of linear systems with arbitrary disturbances, that results in a scalable solution based in convex programming. 

Enforcing probabilistic safety constraints in stochastic optimal control problems is difficult because it typically requires
high dimensional integrals that are hard to compute and enforce. 
The two main approaches to tackle chance constraints are based in sampling or risk allocation~\cite{mesbah2016stochastic}. 
Sampling based approaches
approximate the uncertainty distribution using a finite
number of samples (particles), and formulate a mixed-integer
optimization problem~\cite{blackmore2011chance}. This approach is
independent of the particular distribution, and has well characterized lower
bounds on the number of particles needed to achieve high
quality
solutions~\cite{calafiore2006scenario,SartipizadehACC2019}.
However, these bounds typically require a large
number of particles, resulting in computationally expensive, mixed-integer optimization problems. 

In contrast, risk allocation based approaches are sampling-free  approaches that compute open-loop or affine-feedback controllers~\cite{oldewurtel2014stochastic,ono2008iterative,vitus_feedback_2011,paulson2017stochastic}. They utilize Boole's inequality to decompose joint chance constraints into simpler, individual chance constraints, and optimize for violation probability thresholds present in the constraints.
For a fixed risk allocation, the control synthesis problem is convex for Gaussian disturbances~\cite{oldewurtel2014stochastic,ono2008iterative}. On the other hand, non-Gaussian disturbances admit convex but conservative enforcement of the chance constraints using concentration inequalities~\cite{paulson2017stochastic,CalafioreJOTA2006}.
Unfortunately, simultaneous risk allocation and controller synthesis renders the optimal control problem non-convex.
Therefore, existing approaches leverage coordinate descent algorithms to approximately solve the stochastic optimal control problem.

{\em Our main contribution is a computationally efficient and numerically robust solution for stochastic optimal control of linear dynamical systems with non-Gaussian disturbances, based in risk allocation, Fourier transformations, and convex optimization.} Our approach simultaneously performs risk allocation and open-loop controller synthesis, without compromising on computational tractability or relying on conservative enforcement of chance-constraints.
The key to this is 1) the use of
characteristic functions (Fourier transformations of the probability density function) to enforce chance constraints involving non-Gaussian random vectors exactly, and 2) 
reformulation of the risk allocation problem
as a difference-of-convex program, which  
can be solved locally efficiently 
via convex optimization~\cite{lipp_variations_2016}.
In combination with tight, conic, piecewise affine approximations of the non-conic convex constraints, we can 
leverage standard off-the-shelf conic solvers to solve the stochastic optimal control problem.  

The main limitation of this approach is that it requires open-loop controller synthesis, which results in more conservative solutions than with a closed-loop controller.  
Open-loop control synthesis are commonplace in stochastic model predictive control~\cite{mesbah2016stochastic,lorenzen2016constraint}, and 
essential in applications with hard computational constraints or sensing constraints that preclude feedback control.  
Consider 
hypersonic vehicles, which suffer from computing and sensing limitations at their operational speeds and temperatures \cite{chan1991aeroservoelastic,parker2007control}, or space applications in harsh environments, such as on Mars~\cite{balaram2018mars}, in which production and testing of sensors that work reliably is difficult.

The organization of the paper is as
follows: 
We present the problem formulation in Section~\ref{sec:PS}.
Reformulation of the stochastic optimal control problem
using risk allocation, piecewise affine approximation, and difference-of-convex programming is presented in Section~\ref{sec:convexification}. 
Specialization to Gaussian disturbances, and to random
initial conditions are presented in Section~\ref{sec:extensions}. 
We demonstrate our approach on two motion planning examples in Section~\ref{sec:examples}, and summarize our contribution in Section~\ref{sec:conclusion}.

\section{Problem statement}
\label{sec:PS}

We employ the following notation throughout the paper:  The discrete-time interval
$\mathbb{N}_{[a,b]}$ enumerates all natural numbers
from integers $a$ to $b$.  Random vectors are denoted with a bold case
$\bv$, non-random vectors are denoted with an overline $\overline{v}$,
and the trace operator is denoted by $ \mathrm{tr}(\cdot)$.

Consider a stochastic, linear, time-varying system
\begin{align}
    \bx(k+1)&=A(k)\bx(k) + B(k)\overline{u}(k) + \bw(k)\label{eq:sys}
\end{align}
with state $\bx(k)\in \mathbb{R}^n$, input
$\overline{u}(k)\in \mathcal{U}\subset\mathbb{R}^m$, and
disturbance $\bw(k)\in \mathbb{R}^p$. For a time horizon of
$N\in \mathbb{N}$, we assume knowledge of the
disturbance probability density $\psi_{\bW}$ describing the
stochasticity of the concatenated disturbance random vector
$\bW={[\bw(0)^\top\ \bw(1)^\top\ \ldots\
\bw(N-1)^\top]}^\top\in \mathbb{R}^{pN}$. 
For example, for an independent and identical random
disturbance process $\bw(k)\sim\psi_{\bw}$ with $k\in
\mathbb{N}_{[0,N-1]}$,
$\psi_{\bW}=\prod_{k=0}^{N-1}\psi_{\bw}$.

Throughout the paper, we will assume that $\psi_{\bW}$ is log-concave. 
Log-concave probability densities form a wide class of
unimodal densities~\cite{dharmadhikari1988unimodality},
including Gaussian and exponential disturbances, and
disturbances with {convex finite support} like
triangular and uniform disturbances over convex sets.
Recall that a function $f: \mathbb{R}_{\geq 0} \to
\mathbb{R}$ is log-concave, if $\log (f)$ is a
concave~\cite[Sec.  3.5.1.]{BoydConvex2004}. We follow the
convention that $\log(0)\triangleq -\infty$. Since
log-concavity is preserved under products, log-concavity of
$\psi_{\bw_k}$ is sufficient for log-concavity of
$\psi_{\bW}$.

Given a fixed initial state $\overline{x}(0) \in
\mathbb{R}^n$, we define the concatenated (stochastic) state vector and
concatenated (deterministic) input vector associated with the
dynamics \eqref{eq:sys} as follows:
\begin{subequations}
    \begin{align}
        \bX&={\left[{\bx(1)}^\top\ \ldots\
        {\bx(N)}^\top\right]}^\top\in \mathbb{R}^{nN}
        \label{eq:concat_X},\\
    \overline{U}&={\left[{\overline{u}(0)}^\top\ \ldots\
    {\overline{u}(N-1)}^\top\right]}^\top\in
    \mathcal{U}^{N}\subset\mathbb{R}^{mN}
    \label{eq:concat_U}.
    \end{align}\label{eq:concat}%
\end{subequations}%
From \eqref{eq:sys} and \eqref{eq:concat}, we have
\begin{align}
    \bX&= \bar{A} \overline{x}(0) + H \overline{U} + G \bW\label{eq:stacked_X}
\end{align}
where the matrices $ \bar{A}\in \mathbb{R}^{nN\times
n}$, $H\in \mathbb{R}^{nN\times mN}$, and $G\in
\mathbb{R}^{nN\times pN}$ are obtained from the dynamics
\eqref{eq:sys}. Due to the linearity of
\eqref{eq:stacked_X}, the mean and the covariance vector of
$\bX$ admit closed-form expressions, 
\begin{subequations}
\begin{align}
    \mubXU&=\bar{A} \overline{x}(0) + H \overline{U} + G \mubW\label{eq:X_U_mu}\\
    \CbXU&=G\CbW G^\top\label{eq:X_U_cov}.
\end{align}\label{eq:X_U_stoch}%
\end{subequations}%

We are interested in solving a stochastic optimal control
problem that minimizes a quadratic cost in $\bX$ and $
\overline{U}$ with
pre-specified positive semi-definite matrices $Q\in
\mathbb{R}^{(nN)\times(nN)}$ and $R\in
\mathbb{R}^{(mN)\times (mN)}$, while satisfying hard constraints
on the input $\mathcal{U}^{N}\subset \mathbb{R}^{mN}$, and soft constraints on the state with high
probability. 
We assume that the input and state constraints are polytopic. 
Given a {probabilistic constraint violation
threshold} $\Delta\in[0,1)$ and a desired trajectory $\Xd
\in \mathscr{S}$, we wish to solve the following stochastic
optimal control problem,
\begin{subequations}
    \begin{align}
        \underset{\overline{U}}{\mathrm{minimize}}&\quad \mathbb{E}_{\bX}[(\bX - \Xd)^\top Q(\bX -
\Xd) + \overline{U}^\top R \overline{U}]\label{eq:stoc_cost}\\
        \mathrm{subject\ to}
   						    &\quad \overline{U}\in
                            \mathcal{U}^{N}, \eqref{eq:X_U_mu}, \eqref{eq:X_U_cov}\label{eq:stoc_input}\\
   						    &\quad \ProbbXU\left\{\bXU\in \mathscr{S}\right\}\geq 1-\Delta\label{eq:stoc_sc}
    \end{align}\label{prob:stoc}%
\end{subequations}%
with decision variable $ \overline{U}\in
\mathbb{R}^{mN}$. 
The cost function is {convex quadratic} in $
\overline{U}$ since $\mathrm{tr}(Q \CbXU)$ is independent of
$ \overline{U}$ by \eqref{eq:X_U_cov}.
We define $
\mathscr{S}=\{ \overline{X}\in \mathbb{R}^{nN}: P
\overline{X}\leq \overline{q}\}$ with $P={\left[
\overline{p}_1^\top\ \ldots\
\overline{p}_{L_X}^\top\right]}^\top\in
\mathbb{R}^{{L_X}\times nN}$ and $\overline{q}={[q_1\
\ldots\ q_{L_X}]}^\top \in \mathbb{R}^{L_X}$ with $L_X\in
\mathbb{N}$ defining the number of hyperplanes in the polytope.

For a $\psi_{\bw}$ that is Gaussian, risk allocation is an established approach to conservatively assure (5c)~\cite{oldewurtel2014stochastic,ono2008iterative,vitus2016stochastic,vitus_feedback_2011,vinod2019piecewise}.
By exploiting the properties of a Gaussian random variable, in conjunction with Boole's inequality, (5c) can be reformulated as a collection of linear or second order cone constraints. 
This results in a convex program which enables efficient controller synthesis via standard solvers. 

However, non-Gaussian disturbances do not admit similar
reformulations. For non-Gaussian disturbances, particle based and
moment based approaches are the two main approaches to solve
\eqref{prob:stoc}. However, these approaches have
significant drawbacks. Particle based approaches use
sampling to approximate \eqref{eq:CDF_cc_a}, and rely on
computationally expensive, mixed integer, linear program
solvers for controller
synthesis~\cite{blackmore2011chance,SartipizadehACC2019}. 
Moment based approaches use concentration
inequalities and risk allocation to enforce
\eqref{eq:stoc_sc}. Even though the moment based approaches
enable controller synthesis via convex optimization, the
resulting reformulation is typically conservative \cite{mesbah2016stochastic,paulson2017stochastic,nemirovski2006convex}.
The conservativeness arises from the fact that only few
lower-order moments are used to tractably enforce the chance
constraints, ignoring the available, higher-order moment
information.

To address the computationally expensive nature of the
particle based control and the conservativeness of the
moment based approach, we present a Fourier transform based
approach to solve \eqref{prob:stoc}, 
which uses all the
moments of the underlying distribution. 
We propose to solve two
problems:

\begin{prob}
    Extend the risk-allocation technique for non-Gaussian
    disturbances using Fourier transforms and
    piecewise affine approximations.
\end{prob}
\begin{prob}
    Solve \eqref{prob:stoc} for an arbitrary, log-concave,
    stochastic disturbance $\bW$ using convex optimization
    and piecewise affine approximation of the chance constraint from Problem 1.
\end{prob}

\section{Convexification of non-Gaussian joint chance constraints}
\label{sec:convexification}
\subsection{Risk-allocation for log-concave sisturbances}

The standard risk-allocation approach~\cite{vitus2016stochastic,oldewurtel2014stochastic,ono2008iterative,vitus_feedback_2011,vinod2019piecewise}, transforms the joint chance constraints
\eqref{eq:stoc_sc} into a set of {individual chance
constraints} via Boole's inequality,
\begin{align}
        &\mathbb{P}\left\{P\bXU\leq\overline{q}\right\} \geq 1 - \Delta\\
    \Leftrightarrow &\mathbb{P}\left\{\cap_{i=1}^{L_X}
        \left\{p_i^\top G\bW\leq q_i -
        \overline{p}_i^{\top}\left(\bar{A} \overline{x}(0) + H
\overline{U}\right)\right\}\right\} \geq 1 - \Delta \nonumber \\
    \Leftrightarrow &\mathbb{P}\left\{\cup_{i=1}^{L_X}
        \left\{p_i^\top G\bW > q_i -
        \overline{p}_i^{\top}\left(\bar{A} \overline{x}(0) + H
\overline{U}\right)\right\}\right\} \leq \Delta \nonumber \\
    \Leftarrow &\sum_{i=1}^{L_X}\mathbb{P}\left\{
        p_i^\top G\bW > q_i -
        \overline{p}_i^{\top}\left(\bar{A} \overline{x}(0) + H
\overline{U}\right)\right\} \leq \Delta \nonumber \\
    \Leftarrow & \left\{
        \begin{array}{l}
            \mathbb{P}\left\{p_i^{\top}G\bW\leq
            q_i- \overline{p}_i^{\top}\left(\bar{A}
        \overline{x}(0) + H \overline{U}\right)\right\}\\
        \hspace*{8em}
        \geq 1 - \delta_i,\ \forall i \in
        \mathbb{N}_{[1,L_X]},\\[1ex]
        \sum\nolimits_{i=1}^{L_X}
        \delta_i \leq\Delta,\ \delta_i \in [0,\Delta],\ \forall i
        \in \mathbb{N}_{[1,L_X]}
        \end{array}\right.\label{eq:cc}.
\end{align}
Here, $\delta_i \in[0,1)$ are auxiliary decision variables
that represent the risk of violating the constraint
$p_i^\top \overline{X} \leq q_i$, $i\in
\mathbb{N}_{[1,L_X]}$. We have $\delta_i \leq \Delta$ since
$\sum_{i=1}^{L_X}\delta_i\leq \Delta$ and $\delta_i$ are
non-negative.

Let $\Phi_{ \overline{p}_i^{\top}G\bW}: \mathbb{R} \to
[0,1]$ denote the cumulative distribution function of the
random variable $ \overline{p}_i^{\top}G\bW$,
\begin{align}
\Phi_{
\overline{p}_i^{\top}G\bW}\left(q'\right)=\mathbb{P}\left\{\overline{p}_i^{\top}G\bW\leq
q'\right\},
\end{align}
for any scalar $q' \in \mathbb{R}$. We use $\Phi_{
\overline{p}_i^{\top}G\bW}$ to rewrite the
constraints \eqref{eq:cc} as
\begin{subequations}
\begin{align}
    &\Phi_{ \overline{p}_i^{\top}G\bW}\left(d_i-
 \overline{p}_i^{\top}H
\overline{U}\right) \geq  1-\delta_i&\forall i \in
    \mathbb{N}_{[1,L_X]},\label{eq:CDF_cc_a}\\
                                     &\sum\nolimits_{i=1}^{L_X}
\delta_i\leq\Delta ,\ \delta_i\in [0,\Delta],&\forall i \in
\mathbb{N}_{[1,L_X]},\label{eq:CDF_cc_b}
\end{align}\label{eq:CDF_cc}%
\end{subequations}%
with scalar constants $$d_i = q_i-
\overline{p}_i^{\top}\bar{A} \overline{x}(0),\ \forall i\in
\mathbb{N}_{[1,L_X]}.$$
Any feasible controller $\overline{U} \in \mathcal{U}^N$
with a feasible risk allocation
$\overline{\delta}\triangleq [\delta_1\ \cdots\
\delta_{L_X}]\in {[0,1]}^{L_X}$ that satisfies
\eqref{eq:CDF_cc} automatically satisfies \eqref{eq:stoc_sc}.

\subsection{Enforcing chance constraints using characteristic functions}
\label{sub:cf}

The characteristic function of the disturbance vector $\bW$
with probability density function $\psi_{\bW}(\bar{z})$ is
defined as 
\begin{align}
    \Psi_{\bW}(\bar{\beta})&\triangleq
    \Exp_{\bW}\left[\mathrm{exp}\left({j\bar{\beta}^\top\bW}\right)\right] \nonumber \\
    &=\int_{\mathbb{R}^p}\exp({j\bar{\beta}^\top\bar{z})}
    \psi_{\bW}(\bar{z})d\bar{z}=
    \mathscr{F}\left\{\psi_{\bW}\right\}(-\bar{\beta})\label{eq:cfun_def}
\end{align}
where $ \mathscr{F}\{\cdot\}$ denotes the Fourier
transformation operator and $\bar{\beta}\in
\mathbb{R}^{nN}$.
Furthermore, from~\cite[Eq.  22.6.3]{cramer2016mathematical}, the
characteristic function of the random variable $
\overline{p}_i^{\top}G\bW$ is given by
\begin{align}
    \Psi_{ \overline{p}_i^{\top}G\bW}(\beta)=\Psi_{\bW}(
    (G^\top \overline{p}_i) \beta)\label{eq:Psi_aw}
\end{align}
for some $\beta \in \mathbb{R}$.

A key insight we use in this paper is that the evaluation of
the cumulative distribution function in~\eqref{eq:CDF_cc_a}
is given by a one-dimensional integration, i.e., for any
$s\in \mathbb{R}$,
\begingroup
    \makeatletter\def\f@size{9}\check@mathfonts
\begin{align}
    \Phi_{\overline{p}_i^\top G\bW}(s) &=
    \frac{1}{2}-\frac{1}{2\pi}\int_{0}^{\infty}\mathrm{Im}\left(\frac{\exp(j\beta
            s)\Psi_{\bW}((G^\top
    \overline{p}_i)\beta )}{j\beta}\right) d\beta,\label{eq:G_FT}
\end{align}
\endgroup%
where $\mathrm{Im}(z)$ denotes the imaginary component of a
complex number $z$.  Equation \eqref{eq:G_FT} enables
enforcing the chance constraint in \eqref{eq:CDF_cc_a} using
only $\Psi_{\bW}$ as opposed to using the probability density function, 
the known characteristic function of the
concatenated disturbance random vector $\bW$.  
Equation \eqref{eq:G_FT} follows from the
inversion of characteristic
functions~\cite{Gil,waller,witov}. We implement
\eqref{eq:G_FT} using quadrature techniques~\cite{cfun}.

\begin{lem}[$\text{\cite[Thm. 4.2.1]{prekopa1995stochastic}}$]\label{lem:logconcave_cdf}
    Under the assumption of log-concavity, $\Phi_{\overline{p}_i^{\top}G\bW}$ is log-concave over $ \mathbb{R}$.
\end{lem}
Using \eqref{eq:CDF_cc} and Lemma~\ref{lem:logconcave_cdf},
we approximate \eqref{prob:stoc} as follows,
\begin{subequations}
    \newcommand{\optSpace}{0.6em}
    \begin{align}
        \underset{\overline{U},\overline{t}}{\mathrm{minimize}}&\hspace*{\optSpace}  {(\mubXU - \Xd)}^\top Q{(\mubXU -  \Xd)}+ \overline{U}^\top R \overline{U}\nonumber\\
        &\quad + \mathrm{tr}(Q \CbXU)\label{eq:stoc_ng_cost}\\
        \mathrm{subject\ to}
        &\hspace*{\optSpace} \overline{U}\in
        \mathcal{U}^{N}\label{eq:stoc_ng_input}\\
\forall i\in \mathbb{N}_{[1,L_X]}, &\hspace*{\optSpace}
    \overline{p}_i^\top H\overline{U} +
    \Phi_{\overline{p}_i^\top G\bW}^{-1}(\epsilon) \leq d_i \label{eq:stoc_ng_nonneg}\\
\forall i\in\mathbb{N}_{[1,L_X]}, &\hspace*{\optSpace}
\log\left(\Phi_{\overline{p}_i^\top G\bW}(d_i -
    \overline{p}_i^\top H\overline{U})\right)\geq t_i \label{eq:stoc_ng_ICC}\\
            \forall i\in\mathbb{N}_{[1,L_X]},
                                  &\hspace*{\optSpace}
                                  t_i\in[\log(1-\Delta),0]\label{eq:stoc_ng_t_i}\\
        \forall i\in\mathbb{N}_{[1,L_X]}, & \hspace*{\optSpace}
        \log\left(\sum_{i=1}^{L_X}\exp(t_i)\right)\geq
        \log(L_X-\Delta).\label{eq:stoc_ng_ralloc}
        \end{align}\label{prob:stoc_ng}%
\end{subequations}%
for a small scalar $\epsilon > 0$ and a change of
variables 
\begin{align}
    t_i\triangleq\log(1-\delta_i),\ \forall i \in
    \mathbb{N}_{[1,L_X]}\label{eq:ti_defn}
\end{align}
with $\overline{t}=[t_1,...,t_{L_X}]\in\mathbb{R}^{L_X}$.


We now establish the relationship between
\eqref{prob:stoc} and \eqref{prob:stoc_ng}, and show that
\eqref{prob:stoc_ng} is a non-convex program with a reverse
convex constraint. Recall that reverse-convex constraints
are optimization constraints of the form $f \geq 0$, where
$f$ is a convex function.

\begin{thm}\label{thm:feas1}
    Assuming that the underlying distribution is log-concave, the
    following statements hold for any $\Delta\in[0,1)$ and
    any $\epsilon > 0$:
    \begin{enumerate}
        \item Every feasible solution of \eqref{prob:stoc_ng}
            is feasible for~\eqref{prob:stoc}, and
        \item The cost and the constraints
            \eqref{eq:stoc_ng_input}--\eqref{eq:stoc_ng_ICC}
            are convex.  However, \eqref{eq:stoc_ng_ralloc}
            is a reverse convex constraint.
    \end{enumerate}
\end{thm}
\begin{proof}
    \emph{1)} We observe that the constraints
    \eqref{eq:stoc_input} and \eqref{eq:stoc_ng_input} are
    identical. We need to show that satisfaction of
    \eqref{eq:stoc_ng_nonneg}--\eqref{eq:stoc_ng_ralloc}
    satisfies \eqref{eq:stoc_sc}. Recall that the collection
    of constraints \eqref{eq:CDF_cc} tighten
    \eqref{eq:stoc_sc}. Therefore, it is sufficient to show
    that the satisfaction of constraints
    \eqref{eq:stoc_ng_nonneg}--\eqref{eq:stoc_ng_ralloc}
    guarantee satisfaction of \eqref{eq:CDF_cc}.

    The constraint \eqref{eq:stoc_ng_nonneg}
    ensures that the constraint \eqref{eq:stoc_ng_ICC} is
    well-defined, since the satisfaction of
    \eqref{eq:stoc_ng_nonneg} ensures that
    $\Phi_{\overline{p}_i^\top G\bW}(d_i -
    \overline{p}_i^\top H\overline{U})$ is
    positive. The satisfaction of
    \eqref{eq:stoc_ng_ICC} implies satisfaction of
    \eqref{eq:CDF_cc_a}. The satisfaction
    of \eqref{eq:stoc_ng_t_i} implies that $\delta_i \in
    [0,\Delta]$ by \eqref{eq:ti_defn}.
    Finally, 
    we show that \eqref{eq:stoc_ng_ralloc} and
    \eqref{eq:CDF_cc_b} are equivalent via simple
    algebraic manipulations,
    \begin{subequations}
    \begin{align}
         \sum_{i=1}^{L_X}\delta_i \leq
         \Delta&\Leftrightarrow L_X -
         \sum_{i=1}^{L_X}(1-\delta_i) \leq \Delta\\
               &\Leftrightarrow
               \log\left(\sum_{i=1}^{L_X}\exp(t_i)\right)
               \geq \log\left(L_X - \Delta\right)
    \end{align}
    \end{subequations}
    In other words, every feasible solution $(\overline{U},
    \overline{t})$ of \eqref{prob:stoc_ng} maps to a
    feasible solution to \eqref{eq:CDF_cc} with $
    \overline{U}\in \mathcal{U}^N$, and thereby is
    feasible for \eqref{prob:stoc}.

    \emph{Proof of 2)} We already know that the cost
    \eqref{eq:stoc_ng_cost} is a convex quadratic function
    of $ \overline{U}$. By construction, the
    constraints \eqref{eq:stoc_ng_input},
    \eqref{eq:stoc_ng_nonneg}, and \eqref{eq:stoc_ng_t_i}
    are linear constraints in $ \overline{U}$ and $
    \overline{t}$. The convexity of \eqref{eq:stoc_ng_ICC}
    follows from Lemma~\ref{lem:logconcave_cdf} and the
    definition of log-concavity. Recall that
    $\log\left(\sum_{i=1}^{L_X}\exp(t_i)\right)$ is a convex
    function in $\overline{t}$~\cite[Sec.
    3.1.5]{BoydConvex2004}, which shows that
    \eqref{eq:stoc_ng_ralloc} is a reverse-convex
    constraint.
\end{proof}

\subsection{Conic reformulation of \eqref{eq:stoc_ng_ICC} via piecewise affine approximation}
\label{sub:pwa_conic}

We now focus on enforcing the
convex constraint \eqref{eq:stoc_ng_ICC}. Despite its
convexity, the constraint \eqref{eq:stoc_ng_ICC} is not a
conic constraint, which prevents the use of
standard conic solvers in its current form. We present a
tight conic reformulation of \eqref{eq:stoc_ng_ICC} using
piecewise affine approximations. 

Given a concave function $f: \mathcal{D} \to \mathcal{R}$
for bounded intervals $ \mathcal{D}, \mathcal{R} \subset \mathbb{R}$, we
define its piecewise affine underapproximation as
$\ell_f^-: \mathbb{R} \rightarrow \mathbb{R}$ for some
$\UPWAm_{j},\UPWAc_{j}\in \mathbb{R}$ for $j\in
\mathbb{N}_{[1,N_f]}$ and $N_f \in \mathbb{N}$ distinct affine elements,
\begin{equation}
    \ell^-_{f}(x)\triangleq \min_{j\in \mathbb{N}_{[1,N_f]}}
    (\UPWAm_{j} x + \UPWAc_{j}) \label{eq:ell_defn_f}.
\end{equation}

\begin{figure}
    \centering
    \includegraphics[width=\linewidth]{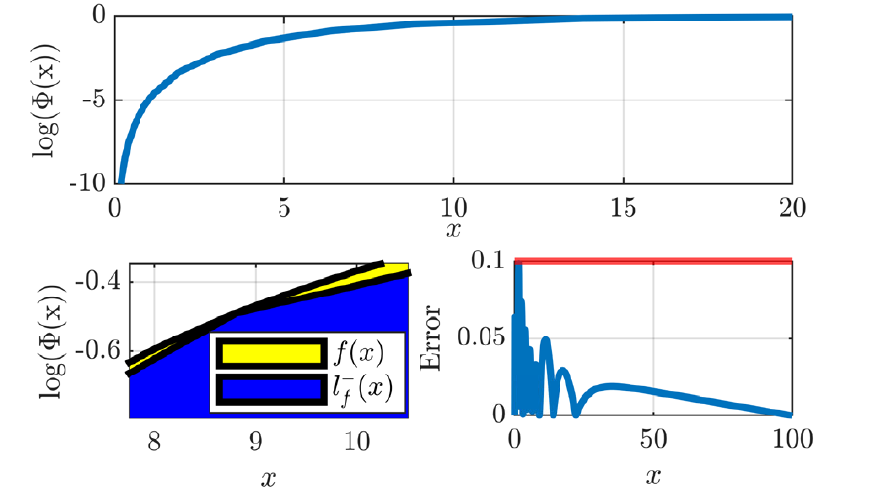}
    \caption{A piecewise affine underapproximation (blue) of the log of the cumulative distribution function of an affine transformation of a random vector (yellow) $a^{\top}\bw_t$ where $\bw_t = [\bw_1\ \bw_2\ \bw_3]^{\top}\in\mathbb{R}^3$ where the scale parameters are $\overline{\lambda}_{\bw}(k) = {[0.5\ 0.25\ 0.1667]}^\top$. The underapproximation is obtained via the sandwich algorithm (Appendix~\ref{ap:pwa}).}
    \label{fig:PWAFit}
\end{figure}

For a user specified approximation error $\eta > 0$,
Appendix~\ref{ap:pwa} describes the {sandwich
algorithm}~\cite{rote1992convergence} that computes
$\ell_f^-$ for a concave $f$ such that
\begin{equation}
    \ell_f^-(x)\leq f(x) \leq \ell_f^-(x) + \eta\label{eq:ell_defn_approx}.
\end{equation}

In \eqref{prob:stoc_ng}, we use the piecewise affine
{underapproximation} of the concave functions $f_i=\log\left(\Phi_{
\overline{p}_i^\top G\bW}\right)$  with $N_i \in \mathbb{N}$
distinct pieces for every $i\in
\mathbb{N}_{[1, L_X]}$ to conservatively enforce
\eqref{eq:stoc_ng_ICC}. The functions $f_i$ have
bounded domain and range in $\mathbb{R}$ due to 
\eqref{eq:stoc_ng_nonneg}. We evaluate $f_i$
using the one-dimensional numerical integration of
characteristic functions, as discussed in \eqref{eq:G_FT}.
We obtain the following optimization problem, 
\begin{subequations}\label{prob:stoc_dc}
    \newcommand{\optSpace}{0.6em}
    \begin{align}
        \underset{\overline{U},\overline{t}}{\mathrm{minimize}}&\hspace*{\optSpace}  {(\mubXU - \Xd)}^\top Q{(\mubXU -  \Xd)}+ \overline{U}^\top R \overline{U}\nonumber\\
        &\quad + \mathrm{tr}(Q \CbXU)\label{eq:stoc_dc_cost}\\
        \mathrm{subject\ to}
        &\hspace*{\optSpace} \overline{U}\in
        \mathcal{U}^{N}\label{eq:stoc_dc_input}\\
\forall i\in \mathbb{N}_{[1,L_X]}, &\hspace*{\optSpace}
    \overline{p}_i^\top H\overline{U} +
    \Phi_{\overline{p}_i^\top G\bW}^{-1}(\epsilon) \leq d_i \label{eq:stoc_dc_nonneg}\\
\substack{\forall i\in\mathbb{N}_{[1,L_X]}\\\forall
j\in\mathbb{N}_{[1,N_i]}}, &\hspace*{\optSpace}
\UPWAm_{i,j} \left({d_i -
    \overline{p}_i^\top H\overline{U}}\right) + \UPWAc_{i,j}\geq t_i \label{eq:stoc_dc_ICC}\\
        \forall i\in\mathbb{N}_{[1,L_X]},
                                  &\hspace*{\optSpace}
                                  t_i\in[\log(1-\Delta),0]\label{eq:stoc_dc_t_i}\\
        & \hspace*{\optSpace}
        \log\left(\sum_{i=1}^{L_X}\exp(t_i)\right)\geq
        \log(L_X-\Delta).\label{eq:stoc_dc_ralloc}
        \end{align}%
\end{subequations}%
By Theorem~\ref{thm:feas1} and the use of piecewise affine
underapproximations of $\log(\Phi_{ \overline{p}_i^\top
G\bW})$, every feasible solution of \eqref{prob:stoc_dc} is
feasible for \eqref{prob:stoc_ng}, and thereby
\eqref{prob:stoc}. 

\subsection{Solving \eqref{prob:stoc_dc} via difference of
convex programming}
\label{sub:dc}

The optimization problem \eqref{prob:stoc_dc} has a
quadratic cost \eqref{eq:stoc_dc_cost}, linear constraints
\eqref{eq:stoc_dc_input}--\eqref{eq:stoc_dc_t_i} in the decision variables
$ \overline{U}$ and $ \overline{t}$, and a single
reverse-convex constraint \eqref{eq:stoc_dc_ralloc}. We now
discuss a tractable solution to \eqref{prob:stoc_dc} using
difference of convex
programming~\cite{lipp_variations_2016}. 

Difference of convex programs are non-convex optimization
problems of the form,
\begin{align}
    \begin{array}{rl}
        \underset{x \in \mathbb{R}^n}{\mathrm{minimize}}&\quad f( \overline{x})
        - g( \overline{x})\\
        \mathrm{subject\ to}&\quad f_i( \overline{x}) - g_i(
        \overline{x}) \leq 0,\quad \forall i\in \mathbb{N}_{[1,M]}\\
    \end{array}\label{prob:dc_ex}
\end{align}
where $f, g$ and $f_i,g_i$ are convex for $i\in \mathbb{N}_{[1,M]}$,
$M\in \mathbb{N}$. The {penalty based convex-concave
procedure}~\cite{lipp_variations_2016} solves
\eqref{prob:dc_ex} in a sequential convex optimization based
approach starting from a potentially infeasible initial
guess. See Appendix~\ref{ap:dc} and
\cite{lipp_variations_2016,horst2000introduction} for more details.

Given the current estimate for the risk allocation
$\overline{r}=[r_1\ \ldots\ r_{L_X}]\in
[0,1]^{L_X}$, the penalty based convex-concave
procedure solves the following {convex}
approximation of \eqref{prob:stoc_dc} at every iteration,
\begin{subequations}
    \newcommand{\optSpace}{0.6em}
    \begin{align}
        \underset{\overline{U},\overline{t},
        s}{\mathrm{minimize}}&\hspace*{\optSpace}
        {(\mubXU - \Xd)}^\top Q{(\mubXU -  \Xd)}+
        \overline{U}^\top R \overline{U} \nonumber \\
                             &\hspace*{6em}+ \mathrm{tr}(Q \CbXU)+ \tau_k s\label{eq:stoc_dc_relaxed_cost}\\
        \mathrm{subject\ to}
        &\hspace*{\optSpace} \overline{U}\in
        \mathcal{U}^{N},\ s\geq 0\label{eq:stoc_dc_relaxed_input}\\
\forall i\in \mathbb{N}_{[1,L_X]}, &\hspace*{\optSpace}
\overline{p}_i^\top H\overline{U} +
    \Phi_{\overline{p}_i^\top G\bW}^{-1}(\epsilon) \leq d_i \label{eq:stoc_dc_relaxed_nonneg}\\
\substack{\forall i\in\mathbb{N}_{[1,L_X]}\\\forall
j\in\mathbb{N}_{[1,N_i]}}, &\hspace*{\optSpace}
\UPWAm_{i,j} \left({d_i -
    \overline{p}_i^\top H\overline{U}}\right) + \UPWAc_{i,j}\geq t_i \label{eq:stoc_dc_relaxed_ICC}\\
        \forall i\in\mathbb{N}_{[1,L_X]},
                                  &\hspace*{\optSpace}
                                  t_i\in[\log(1-\Delta),0]\label{eq:stoc_dc_relaxed_t_i}\\
        &
        \begin{array}{l}
        \log\left(\sum_{i=1}^{L_X}\exp(r_i)\right)
        \\[1ex]
            +\frac{1}{\sum_{i=1}^{L_X}\exp(r_i)}\sum_{i=1}^{L_X}\exp(r_i)
            (t_i - r_i)\\[1ex] 
        \hspace*{5em}+ s\geq \log(L_X-\Delta)
        \end{array}\label{eq:stoc_dc_relaxed_ralloc}
        \end{align}\label{prob:stoc_dc_relaxed}%
\end{subequations}%
where $\tau_k\geq 0$ for $k\in \mathbb{N}$ are optimization
hyperparameters. The
constraint \eqref{eq:stoc_dc_relaxed_ralloc} corresponds to the
first-order approximation of the reverse-convex constraint
\eqref{eq:stoc_dc_ralloc}, which is relaxed by a scalar slack
variable $s$. We penalize the slack variable
$s$ in the objective \eqref{eq:stoc_dc_relaxed_cost}. We
know \eqref{prob:stoc_dc_relaxed} is convex, since
\eqref{eq:stoc_dc_relaxed_ralloc} is a linear constraint in
$ \overline{t}$ and $s$, and all other constraints and the
objective are convex (Theorem~\ref{thm:feas1}.b). 

Starting with an
arbitrary risk allocation $ \overline{r}_0 \in
{[0,1]}^{L_X}$, we iteratively
solve \eqref{prob:stoc_dc_relaxed} with monotonically
increasing values of $\tau_k$ to promote feasibility. In the
numerical experiments, we chose a uniform risk allocation $
\overline{r}_0=\frac{\Delta}{L_X} \overline{1}_{L_X}$, where
$\overline{1}_{L_X}$ is a $L_X$-dimensional vector of ones.
See Appendix~\ref{ap:dc} for more details on the sequence
${\{\tau_k\}}_{k\geq 1}$ and the stopping conditions for the
penalty based convex-concave procedure.

In summary, we have decomposed the original stochastic
optimal control problem presented in \eqref{prob:stoc} into a convex quadratic problem, via the steps shown in Figure~\ref{fig:restrict}.  
We first employed risk allocation \eqref{prob:stoc_ng}, then we
converted the non-conic convex constraints present in
\eqref{prob:stoc_ng} into conic convex constraints
using piecewise affine approximations as well as the Fourier
transform. Finally, we utilize difference-of-convex
programming to tackle the remaining reverse convex
constraint \eqref{eq:stoc_ng_ralloc}. Thus, our approach
solves a convex (quadratic) program 
\eqref{prob:stoc_dc_relaxed} 
iteratively to compute a local optimum of \eqref{prob:stoc}.
Figure~\ref{fig:restrict} summarizes the resulting convex
optimization problems.

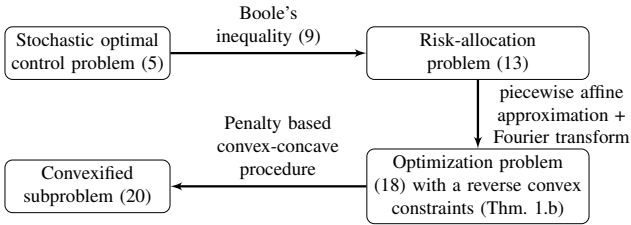
\begin{figure}[ht!]
    \centering
    \adjustbox{width=1\linewidth}{
    \tikzstyle{block} = [rectangle, draw, text width=5em, text centered, rounded corners, minimum height=2em]
    \tikzstyle{line} = [draw, -latex', very thick]
    \tikzstyle{doublearrow} = [draw, latex'-latex', very thick]
    \begin{tikzpicture}[node distance = 10em, auto]
        \node [block, text width=8em] (orig) {Stochastic optimal control problem \eqref{prob:stoc}};
        
        \node [block, text width=11em, xshift=10.5em, right of=orig] (risk) {Risk-allocation problem \eqref{prob:stoc_ng}};
        
        \node [block, text width=11em, yshift= 3em, below of=risk] (log) {Optimization problem \eqref{prob:stoc_dc} with a reverse convex constraints (Thm.~\ref{thm:feas1}.b)};
        
        \node [block, text width=8em, xshift=-10.5em, left of=log] (relaxed) {Convexified subproblem \eqref{prob:stoc_dc_relaxed}};
        \path [line] (orig) -- (risk) node [text centered, above, midway, text width=8em] {Boole's inequality \eqref{eq:CDF_cc}}; 
        
        \path [line] (log) -- (relaxed) node [text centered, above, midway, text width=7.2em, xshift=0.5em] {Penalty based convex-concave procedure};
        
        \path [line] (risk) -- (log) node [text centered,
            right, midway, text width=8em] {piecewise affine
            approximation + Fourier transform};
        
    \end{tikzpicture}}
    \caption{Flow resulting in the convexified problem to solve original problem using standard solvers.}
    \label{fig:restrict}
\end{figure} 

\section{Extensions and special cases}
\label{sec:extensions}

\subsection{Random initial state}

We now consider the effect of a random initial state
$\bx(0)$, which is assumed to be statistically independent
from $\bW$ Similar to \eqref{prob:stoc_dc}, we can use risk
allocation, Fourier transformations, and piecewise affine
approximations to formulate an optimization problem, that
can be solved via penalty based convex-concave procedure.

Let $\Psi_{\bx}$ be the characteristic function of $\bx(0)$.
Define a new random vector $\bZ=\bar{A}\bx(0) + G\bW$. We
have the characteristic function of $\bZ$ in closed-form
with the Fourier variable $ \overline{\beta} \in
\mathbb{R}^{nN}$,
\begin{align}
 \Psi_{\bZ}(\overline{\beta})&=\Psi_{\bx}\left(\bar{A}^\top
    \overline{\beta}\right)\Psi_{\bW}\left(G^\top
\overline{\beta}\right)\label{eq:Psi_bZ}.
\end{align}
Next, we formulate the risk-allocation based constraints on
$\bZ$ to conservatively enforce the soft state constraint
\eqref{eq:stoc_sc}, 
\begin{subequations}
\begin{align}
    \Phi_{ \overline{p}_i^{\top}\bZ}\left(q_i-
 \overline{p}_i^{\top}H
\overline{U}\right) &\geq  1-\delta_i&\forall i \in
    \mathbb{N}_{[1,L_X]},\label{eq:CDF_cc_a_random_x}\\
                                     \sum\nolimits_{i=1}^{L_X}
\delta_i\leq\Delta ,\ \delta_i&\in [0,\Delta],&\forall i \in
\mathbb{N}_{[1,L_X]}.\label{eq:CDF_cc_b_random_x}
\end{align}\label{eq:CDF_cc_random_x}%
\end{subequations}%
Here, we compute $\Phi_{ \overline{p}_i^{\top}\bZ}$ using
\eqref{eq:G_FT} and \eqref{eq:Psi_bZ}. The satisfaction of
\eqref{eq:CDF_cc_random_x} for any feasible controller $
\overline{U}\in \mathcal{U}^N$ and risk allocation $
\overline{\delta}$ implies that $\ProbbXU\{\bX\in
\mathscr{S}\}\geq 1 - \Delta$. In contrast to
\eqref{eq:CDF_cc_a}, \eqref{eq:CDF_cc_a_random_x} has a
different term in the left hand side since the initial state is now
random.

Finally, we complete the optimization problem formulation
using characteristic functions (Sections~\ref{sub:cf}) and piecewise affine underapproximations 
(Section~\ref{sub:pwa_conic}), 
\begin{subequations}
    \newcommand{\optSpace}{0.6em}
    \begin{align}
        \underset{\overline{U},\overline{t}}{\mathrm{minimize}}&\hspace*{\optSpace}  {(\mubXU - \Xd)}^\top Q{(\mubXU -  \Xd)}+ \overline{U}^\top R \overline{U}\nonumber\\
        &\quad + \mathrm{tr}(Q \CbXU)\label{eq:stoc_dc_cost}\\
        \mathrm{subject\ to} &\hspace*{\optSpace}
        \eqref{eq:stoc_dc_input},\ \eqref{eq:stoc_dc_t_i},\
    \eqref{eq:stoc_dc_ralloc}\label{eq:stoc_dc_random_prev}\\
\forall i\in \mathbb{N}_{[1,L_X]}, &\hspace*{\optSpace}
    \overline{p}_i^\top H\overline{U} +
    \Phi_{\overline{p}_i^\top \bZ}^{-1}(\epsilon) \leq q_i\\
\substack{\forall i\in\mathbb{N}_{[1,L_X]}\\\forall
j\in\mathbb{N}_{[1,N_i]}}, &\hspace*{\optSpace}
\UPWAm_{i,j,\bZ} \left({q_i -
    \overline{p}_i^\top H\overline{U}}\right) +
    \UPWAc_{i,j,\bZ}\geq t_i,
        \end{align}\label{prob:stoc_dc_random}%
\end{subequations}%
where $\ell_{f_{\bZ,i}}^-(x)=\min_{j\in
\mathbb{N}_{[1,N_i]}} \left(\UPWAm_{i,j,\bZ} x +
\UPWAc_{i,j,\bZ}\right)$ is the piecewise affine
underapproximation of the concave function
$f_{\bZ,i}(x)=\log(\Phi_{ \overline{p}_i^{\top}\bZ}(x))$.
The optimization problem \eqref{prob:stoc_dc_random} imposes
constraints on the random variable $ \overline{p}_i^\top
\bZ$. In contrast, \eqref{eq:stoc_dc_nonneg} and
\eqref{eq:stoc_dc_ICC} imposed constraints on the random
variable $ \overline{p}_i^\top G\bW$ because the initial
state in \eqref{prob:stoc_dc} was deterministic.

We use sandwich algorithm (Appendix~\ref{ap:pwa}, Algorithm~\ref{algo:approx}) to compute $\ell_{f_{\bZ,i}}^-$.
Similarly to Lemma~\ref{lem:logconcave_cdf}, $\Phi_{
\overline{p}_i^{\top}\bZ}$ is a log-concave function when
$\bx(0)$ and $\bW$ have log-concave probability density
function~\cite[Thm. 4.2.1]{prekopa1995stochastic}.
Similarly to Theorem~\ref{thm:feas1}, the optimization
problem \eqref{prob:stoc_dc_random} has a convex objective
and convex constraints, except for a reverse convex
constraint \eqref{eq:stoc_dc_ralloc} in
\eqref{eq:stoc_dc_random_prev}. Thus,
\eqref{prob:stoc_dc_random} can also be solved using
penalty based convex-concave procedure, similarly to
\eqref{prob:stoc_dc}.


\subsection{Gaussian disturbance $\bW$: Risk allocation and
controller synthesis via a single quadratic program for $\Delta
\leq 0.5$}

For a Gaussian disturbance $\bW$, existing literature solves
the optimal control problem \eqref{prob:stoc} via the
following approximation, 
\begin{subequations}
    \begin{align}
        \underset{\overline{U}, \overline{\delta}}{\mathrm{min}}&\quad {(\mubXU - \Xd)}^\top Q{(\mubXU -  \Xd)}+ \overline{U}^\top R \overline{U}\nonumber\\
        &\quad + \mathrm{tr}(Q \CbXU)\label{eq:Ono_CDFinv_cost}\\
        \mathrm{s.t.}&\quad
        \Delta\geq\sum\nolimits_{i=1}^{L_X} \delta_i, \ 
        \eqref{eq:X_U_mu},\ \eqref{eq:X_U_cov},\ \overline{U}\in \mathcal{U}^{N}\label{eq:Ono_CDFinv_others}\\
                     &\quad \delta_i\in [0,\Delta],
                     \hspace*{7.2em}\forall i\in
                     \mathbb{N}_{[1,{L_X}]},\label{eq:Ono_CDFinv_delta}\\
                            &\ \ \begin{array}{l}
                                \overline{p}_i^\top H
                                \overline{U}\leq q_i -
                                {\Vert \CbXU^\frac{1}{2}
                                \overline{p}_i\Vert}_2\PhiGaussian^{-1}\left(1
                            - \delta_i \right),\\
                            \end{array} \nonumber \\
                            &\hspace*{13em}\forall i\in \mathbb{N}_{[1,{L_X}]},\label{eq:Ono_CDFinv_CDF}
    \end{align}\label{prob:Ono_CDFinv}%
\end{subequations}%
where $\PhiGaussian^{-1}(\cdot)$ is the inverse cumulative
distribution function. The reformulation
\eqref{prob:Ono_CDFinv} is obtained via risk allocation and
Gaussian random vector properties~\cite{ono2008iterative}.
While \eqref{prob:Ono_CDFinv} is known to be convex
when $\Delta\leq 0.5$, existing approaches solve
\eqref{prob:Ono_CDFinv} via coordinate-descent based
approaches, since \eqref{eq:Ono_CDFinv_CDF} is a 
non-conic constraint. 

Similarly to Section~\ref{sub:pwa_conic}, we can use
piecewise affine approximation to tightly approximate the
convex, non-conic constraint \eqref{eq:Ono_CDFinv_CDF}. Let

\begin{equation}
 \ell_f^-=\min_{j\in
    \mathbb{N}_{\left[1,N_{\Phi}\right]}} (\UPWAm_{j} z +
\UPWAc_{j})\leq f_\Phi(z)\triangleq-\PhiGaussian^{-1}(1-z)   
\end{equation}
be the piecewise affine underapproximation of the concave,
differentiable function $f$ with $N_\Phi \in \mathbb{N}$
distinct pieces. We restrict $z\in[\deltalb,
\Delta]$ for some small $\deltalb>0$ to ensure bounded
domain and range for $f$. We can construct $\ell_{f_\Phi}^-$
using the sandwich algorithm (Appendix~\ref{ap:pwa},
Algorithm~\ref{algo:approx}) since $-\PhiGaussian^{-1}(1-z)$
is concave for $z\in[\deltalb,
\Delta]$.  Consequently, any $
\overline{U} \in \mathbb{R}^{mN}$ and $
\overline{\delta}\in\mathbb{R}^{L_X}$ that satisfies
\begin{align}
\overline{p}_i^\top H \overline{U}\leq q_i + {\Vert
                            \CbXU^\frac{1}{2}
                        \overline{p}_i\Vert}_2\left(\UPWAm_{j} \delta_i +
                    \UPWAc_{j}\right),
                    \label{eq:OnoApprox}
\end{align}
for every $ i \in \mathbb{N}_{[1,L_X]}$ and $j \in
\mathbb{N}_{[1, N_{\Phi}]}$ satisfies \eqref{eq:Ono_CDFinv_CDF}. We
obtain a conservative solution to \eqref{prob:stoc} for a
Gaussian disturbance $\bW$ by solving the following
quadratic program,
\begin{subequations}
    \begin{align}
        \underset{\overline{U}, \overline{\delta}}{\mathrm{minimize}}&\quad{(\mubXU - \Xd)}^\top Q{(\mubXU -  \Xd)}+ \overline{U}^\top R \overline{U}\nonumber\\
        &\quad + \mathrm{tr}(Q \CbXU) \label{eq:Ono_CDFinv_pwl_cost}\\
        \mathrm{subject\ to}&\quad 
        \Delta\geq\sum\nolimits_{i=1}^{L_X} \delta_i, \ 
        \eqref{eq:X_U_mu},\ \eqref{eq:X_U_cov},\
        \overline{U}\in
        \mathcal{U}^{N}\label{eq:Ono_CDFinv_pwl_others}\\
   						    &\quad \delta_i\in
                            [\deltalb,\Delta],\hspace{5em}\forall i\in
                            \mathbb{N}_{[1,{L_X}]}\label{eq:Ono_CDFinv_pwl_delta}\\
                            &\quad \overline{p}_i^\top H \overline{U}\leq q_i + {\Vert
                            \CbXU^\frac{1}{2}
                        \overline{p}_i\Vert}_2\left(\UPWAm_{j} \delta_i +
                    \UPWAc_{j}\right), \nonumber \\
                            &\hspace{6em}\forall i\in
                            \mathbb{N}_{[1,{L_X}]},\forall
                            j\in
                            \mathbb{N}_{[1,N_{\Phi}]}\label{eq:Ono_CDFinv_pwl_CDF}.
    \end{align}\label{prob:Ono_CDFinv_pwl}%
\end{subequations}%
In contrast to existing coordinate-descent based approaches,
we can now use standard quadratic program solvers to solve
\eqref{prob:Ono_CDFinv_pwl} efficiently. See our prior
work~\cite{vinod2019piecewise} for more details.

\section{Numerical Examples}
\label{sec:examples}

We apply the proposed approach on two examples: 1) a
stochastic double integrator, and 2) a
quadrotor in a harsh environment, with crosswind. We also
compare the performance of the controller produced by our
approach to: 1) a particle based approach~\cite{blackmore2011chance}, and 2) a moment based approach~\cite{paulson2017stochastic}. 
We measure the performance of the controllers based on the attained cost,
probability of constraint satisfaction, and computational time. We also used a Monte-Carlo simulation with $10^5$ samples for validation.

All computations are done with MATLAB on an Intel Xeon CPU with 3.80 GHz clock rate
and 32GB RAM. We implemented our algorithm and the particle based approach in 
CVX~\cite{cvx} with Gurobi~\cite{gurobi}. We used
\texttt{fmincon} and CVX to implement the moment based
approach.  We used MPT~\cite{MPT3} and
SReachTools~\cite{SReachTools} for the stochastic optimal
control problem formulation.

For the implementation of the proposed approach via
difference-of-convex programming, we set $\tau_{max} = 10000$,
$\tau_{0} = 0.1$, and $\epsilon_{viol} = 1.2$, and for the termination criteria
we used 100 iterations or $\epsilon_{dc} = 1 \times 10^{-6}$. 
For the sandwich algorithm, we chose
$\eta= 0.1$.

The particle based approach constructs an open-loop controller that solves \eqref{prob:stoc} approximately via mixed-integer programming~\cite{blackmore2011chance}. Specifically, we draw samples (particles) of the disturbance random vector $\bW$ and utilize the particle based approximation of the state constraint probability as well as the expected cost to construct a particle based approximation of \eqref{prob:stoc}. This approach recovers the optimal open-loop controller for \eqref{prob:stoc} as the number of particles considered increases, at the penalty of increased computational time. In the numerical experiments, we used $50$ particles, and reported the average from three separate runs.

The moment based approach constructs an affine-feedback controller via coordinate-descent based optimization~\cite{paulson2017stochastic}. It enforces the chance constraint on the state via concentration inequalities, specifically the Chebyshev-Cantelli inequality. The moment based approach utilizes only
the first and the second moment of the disturbance random vector $\bW$, resulting in a high-degree of conservatism compared to the proposed approach. We also use the moment based approach to generate an open-loop controller by setting the gain matrix (a decision variable) to zero.

\subsection{Constrained control of a stochastic double integrator}\label{sub:di}

We first consider a double integrator system,
\begin{equation}
    \bx(k+1) = \begin{bmatrix}
		1 & T_s \\ 
        0 & 1\\
    \end{bmatrix}\bx(k)+ \begin{bmatrix}
  	\frac{ T_s^2}{2}\\ T_s
\end{bmatrix} \overline{u}(k) + \bw(k) \label{eq:doub_int_dyn}
\end{equation} 
with state $\bx(k)\in \mathbb{R}^2$, input set $
\mathcal{U}=[-20,20]$, exponential disturbance $\bw(k)$ with
scale $\overline{\lambda}_{\bw}(k) \in \mathbb{R}^2_+$,
sampling time $T_s = 0.25$s, and initial position
$\overline{x}(0)=[-1\ 0]^\top$.

We seek to solve a constrained optimal control problem subject to dynamics \eqref{eq:doub_int_dyn}, with
quadratic cost \eqref{eq:stoc_cost} that  encodes our desire
to track $\Xd\in
\mathbb{R}^{n{N}}$, penalize high velocities, and minimize
control effort. Specifically, we choose $Q=
\text{diag}([10\; 1])\otimes I_{(nN)\times (nN)}$, $R =
10^{-3} I_{(mN)\times (mN)}$, ${(\Xd)}_t= {[m_r t + c_r\
0]}^\top,\ \forall t\in \mathbb{N}_{[0,N]}$, and set problem
parameters $m_1,m_2,m_r, c_1,c_2, c_r$ as
$0.222,-0.222,-0.111,-5.222,5.222,$ and $2.111$
respectively.
We define the time varying state constraints as
\begin{align}
    \mathscr{T} &= \left\{ (t, \overline{x})\in \mathbb{N}_{[0,N]}\times
    \mathbb{R}^2 : m_1 t + c_1 \leq \overline{x}_1 \leq m_2 t + c_2\right\}.
    \nonumber 
\end{align}
and wish to maintain constraint satisfaction of $90\%$, i.e. $\Delta = 0.1$.

\subsubsection{Constant Time Horizon, Exponential Distribution}\label{subsub:fixedexp}

\begin{figure}[h!]
    \centering
    \includegraphics[width=\linewidth]{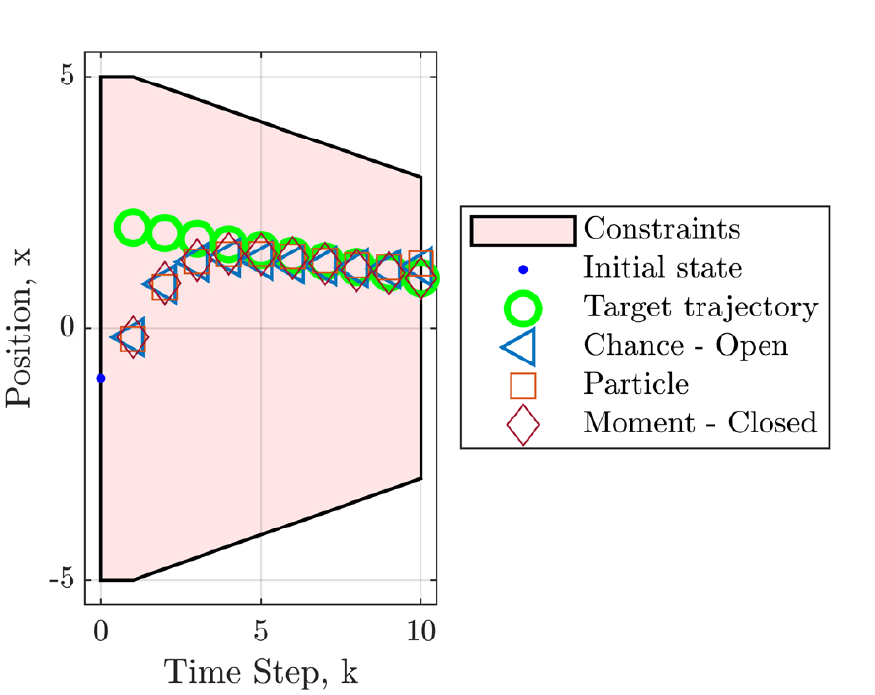}
    \caption{Mean trajectories from our approach, the particle based approach, and the affine feedback moment based approach. All approaches compute a controller that maintains the constraint violation threshold ($\Delta = 0.1$). The affine feedback moment approach tracks the desired trajectory closest, while our approach computes the controller fastest, and has lower constraint violation (Table~\ref{tb:DI}). The open-loop moment based approach failed to find a controller.}
    \label{fig:DI_Mean}
\end{figure}

\begin{figure}
    \centering
    \includegraphics[width=\linewidth]{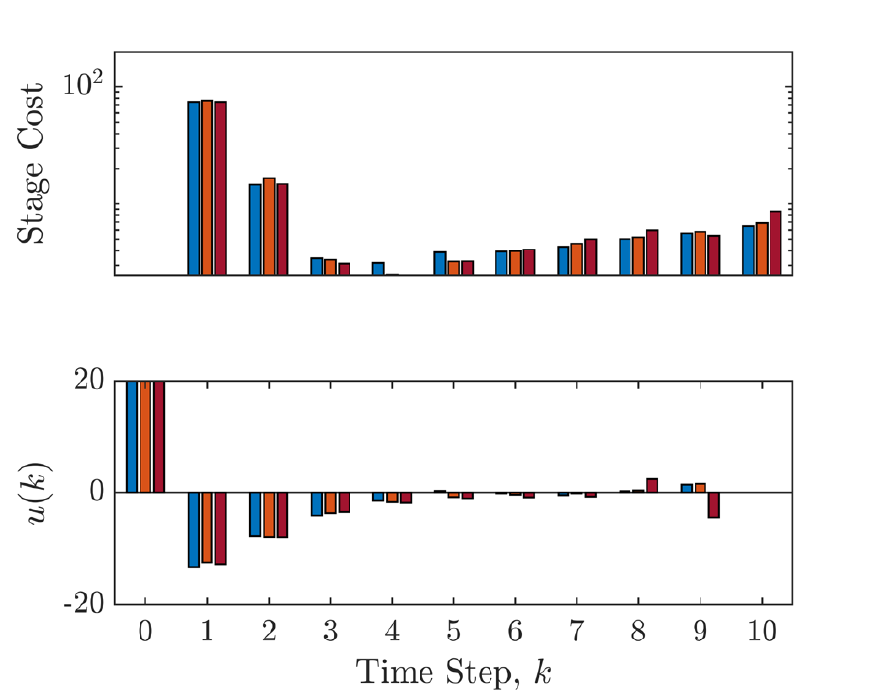}
    \caption{Stage cost (the cost incurred at each time step) and control effort as a function of time for the constant time double integrator. The stage cost for our approach (blue) is higher than the particle based (orange) and moment based affine approach (maroon). The benefit of affine feedback is seen in the control plot for the moment based approach, where $k = 8,9$ show aggressive corrections compared to the other approaches.}
    \label{fig:DI_StageCI}
\end{figure}

We compute optimal control trajectories using our approach, the particle filter and both open and closed loop moment based approaches for a fixed horizon $N=10$ and scale parameter $\overline{\lambda}_{\bw}(k) = {[5\ 10]}^\top$. Figure~\ref{fig:DI_Mean} shows the optimal trajectories for all but the open-loop moment based approach, which failed to compute an optimal trajectory.
Figure~\ref{fig:DI_StageCI} shows that the stage cost (the cost at each time step) is similar amongst both the particle based approach and the affine feedback moment based approach, with a higher cost for our approach. 

\begin{figure}[h!]
    \centering
    \includegraphics[width=\linewidth]{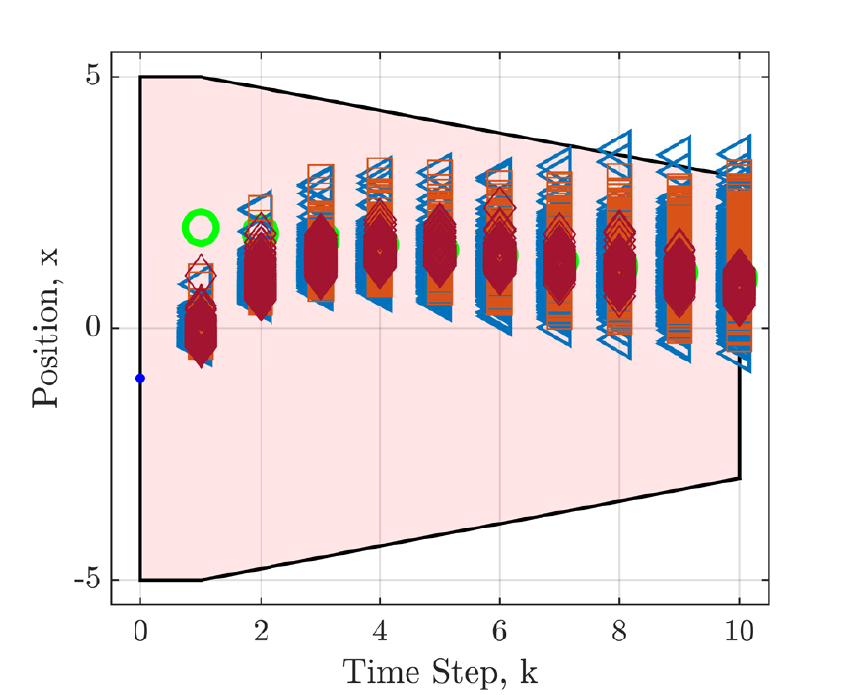}
    \caption{Selected Monte Carlo trajectories for the double integrator.  The affine feedback moment approach (maroon) performs well compared to our approach (blue) and the particle based approach (orange).  However, the computation time is significantly higher than the other approaches (Table~\ref{tb:DI}).  Our approach has less probabilistic constraint violation than the particle based approach.}
    \label{fig:DI_MC}
\end{figure}

While all the methods generated similar trajectories, the key differences can be seen in Table~\ref{tb:DI}, which compares the computed values of the cost and probability of satisfaction to their Monte Carlo estimates for $10^5$ simulated trajectories. 
The particle based approach is able to compute an open-loop controller the fastest using 50 particles, but the constraint violation is lower than the Monte Carlo (MC) estimate of violation. 
On the other hand, the affine feedback moment based approach computes a constraint violation of 0.907, but the Monte Carlo estimate of the constraint violation is 1. 
This is can be seen in Figure~\ref{fig:DI_MC}, which shows a fraction of the Monte Carlo trajectories for all the approaches, where the majority of the affine feedback moment based approach trajectories are well contained in the set.
The benefit of affine feedback is clearly seen in the
control effort in Figure~\ref{fig:DI_StageCI}, where the moment based affine approach provides an input at $k=9$ to maintain a majority of the trajectories.

\begin{table}[t!]
    \centering
    \caption{Double Integrator example: Cost and constraint satisfaction ($1-\Delta$) for both computed (Comp) and Monte Carlo (MC) simulation based validation ($10^5$ samples) for our approach, the particle based approach, and the affine feedback moment based approach. The open-loop moment based approach failed to find a controller. Our approach has higher Monte Carlo constraint satisfaction compared to the particle based approach, and comparable solve time to the affine feedback moment based approach.}
    \begin{tabular}{|p{1.2cm}||c|c|c|c|c|c|}
    \hline
    \multirow{2}{*}{Method}  & \multicolumn{2}{c|}{Cost} & \multicolumn{2}{c|}{ $1-\Delta$} & \multirow{2}{*}{Time (s)} \\\cline{2-5}
    & Comp & MC & Comp & MC & \\\hline\hline
    Chance - Open & 124.599 & 124.507 & 0.90 & 0.981 & 2.468 s\\ \hline
    Particle~\cite{blackmore2011chance} & 108.21 & 108.24 & 1.00 & 0.973 & 1.07 s\\ \hline
    Moment - Closed~\cite{paulson2017stochastic} & 105.628 & 109.482 & 0.907 & 1.00 & 6.88 s\\\hline
    \end{tabular}
    \label{tb:DI}
\end{table}

Our chance constrained approach 
obtains an open-loop controller that exceeds the computed constraint satisfaction in Monte Carlo evaluation, with very little increase in computation time compared to the particle based approach. 
In addition, computation time of our approach is comparable to the affine feedback moment based approach, while having a Monte Carlo estimate of 0.98, providing a balance between a high constraint satisfaction while being cheap to compute.

\subsubsection{Varying time horizons}\label{subsub:di_incrhor}

\begin{figure}
    \centering
    \includegraphics[width=\linewidth]{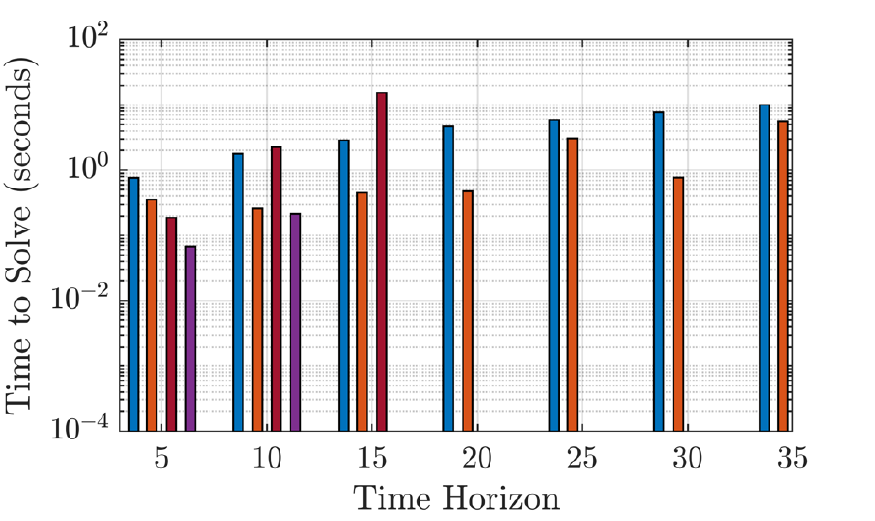}
    \caption{Solve time for varying time horizons for all approaches. Past a horizon of 10 steps
and 15 steps, the open-loop (purple) and affine feedback (maroon) moment based
approaches, respectively, fail to find a controller. Both our
approach (blue) and the particle based approach (orange) are able to compute
a controller.  Our approach has a consistent time to solve. The particle based approach has variability in the time to solve as it is
a mixed integer optimization problem and mixed integer
formulations are in the worst case exponential in terms of
complexity.}
    \label{fig:DI_TTS}
\end{figure}

We compare the solve time as for time horizons between
 0 and 35 time steps, with a exponential disturbance with
scale $\overline{\lambda}_{\bw}(k) = {[10\ 100]}^\top$
(Figure~\ref{fig:DI_TTS}).  

Moment based approaches fail for large time
horizons, possibly due to their reliance on coordinate descent
optimization.  The open-loop approach fails for
time horizons 10 and larger, and the affine approach fails for time horizons 15 and larger. 

While the particle based approach does better than our
approach in solve time
as the time horizon is increased, branch and bound based
approaches are solver specific hence the solve time can vary depending on the solver. 
In addition, as seen for the constant horizon case in Table~\ref{tb:DI}, while the probability of constraint satisfaction for the particle control was noted to be 1, the Monte Carlo estimate was lower.
In comparison, our approach has consistent solve times with Monte Carlo constraint violation greater than what was reported from the computation. 

\subsection{Quadrotor in crosswind of a harsh environment}\label{sub:quad}

We consider a rigid-body quadcopter model, 
\begin{subequations}
\begin{align}
    &\ddot{p}_x = \frac{u_1}{m}\left(\cos\psi\sin\theta+\cos\theta\sin\phi\sin\psi\right)\\
    &\ddot{p}_y = \frac{u_1}{m}\left(\sin\psi\sin\theta-\cos\theta\sin\phi\cos\psi\right)\\
    &\ddot{p}_z = \frac{u_1}{m}\left(\cos\phi\cos\theta\right) -g\\
    &\ddot{\phi} = \frac{I_{yy}-I_{zz}}{I_{xx}}\dot{\theta}\dot{\psi}+\frac{u_2}{I_{xx}}\\
    &\ddot{\theta} = \frac{I_{zz}-I_{xx}}{I_{yy}}\dot{\phi}\dot{\psi}+\frac{u_3}{I_{yy}}\\
    &\ddot{\phi} = \frac{I_{xx}-I_{yy}}{I_{zz}}\dot{\theta}\dot{\phi}+\frac{u_4}{I_{zz}}
\end{align}\label{eq:quadcopter}%
\end{subequations}%
where the state variables $p_x$, $p_y$, and $p_z$ define the translational motion and $\phi,\theta,$ and $\psi$ define the roll, pitch, and yaw respectively. The state is a $12$-dimensional vector, $\overline{x}=[p_x\ p_y\ p_z\ \dot{p}_x\ \dot{p}_y\ \dot{p}_z\ \phi\ \theta\ \psi\ \dot{\phi}\ \dot{\theta}\ \dot{\psi}]^{\top}$. 
The net thrust is described by $u_1$, and the moments around the $p_x$, $p_y$, and $p_z$ axes created by the difference in the motor speeds are described by $u_2$, $u_3$, and $u_4$.
We use the following parameters for the quadcopter: mass $m = 0.478$ kg and moment of inertia $I_{xx} = I_{yy} = 0.0117$ kg m$^2$, and $I_{zz} = 0.00234$ kg m$^2$~\cite{vinod2018multiple}.

We linearize the nonlinear dynamics \eqref{eq:quadcopter} in a hovering operation point (zero state and input $[4.6892,0,0,0]^{\top}$), and discretize the continuous-time dynamics via a zero-order hold with sampling time $T_s= 0.25$. We incorporate the effect of wind into the quadcopter model with an additive stochastic disturbance,
\begin{align}
    \bx(k+1)=A\bx(k) + B\overline{u}(k) + \bw(k)\label{eq:quadcopter_disc}.
\end{align}
We presume a time-invariant triangle distribution to model the wind via the disturbance $\bw(k)$, to characterize the best, worst, and nominal values of the wind (Figure~\ref{fig:distTri}). The wind is assumed to directly influence only the translational motion $p_x$, $p_y$, and $p_z$. The distribution changes in the later half of the control interval as shown in Figure~\ref{fig:distTri}.
\begin{figure}[ht!]
    \centering
    \includegraphics{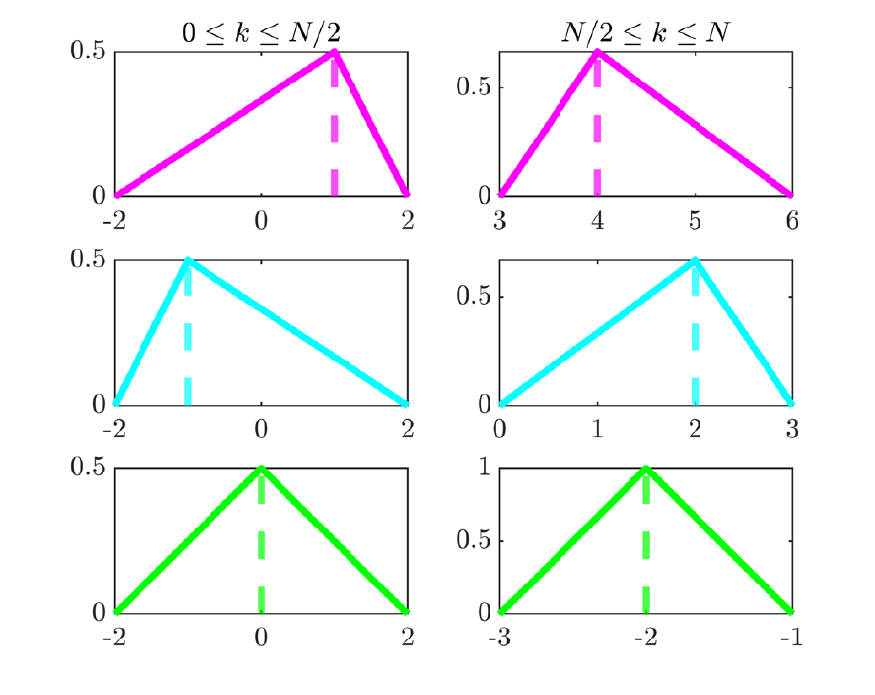}
    \caption{
    The triangular disturbance in the $p_x$, $p_y$, and $p_z$ states of the quadcopter are colored in magenta, cyan, and green respectively. 
    The disturbance starts with the leftmost plots and transitions to the rightmost plots halfway through the time horizon. 
    The parameters of each triangle distribution is given below each plot.}
    \label{fig:distTri}
\end{figure}

We solve the stochastic optimal control problem \eqref{prob:stoc} for a time horizon of $N=10$ with $Q= \text{diag}([10\ 10\ 10\ I_{1\times 9N}])\otimes I_{N\times N}$ and $R = 10^{-3} I_{4N\times 4N}$. 
We specify the desired trajectory $\overline{X}_d$ between $(20,50,25)$ and $(50,20,25)$ via waypoints spread uniformly in time. 
The limits on the input are $\mathcal{U}=[-5,5]^4$. 
The constraint set $\mathscr{S}$, 
\begin{align}
    \mathscr{S} &= \left\{ \overline{x}\in\mathbb{R}^{12}: |p_x|\leq 100,\ |p_y|\leq 100,\ |p_z|\leq 100\right\}
    \nonumber 
\end{align} 
imposes restrictions on the translational motion.
The initial condition is $\overline{x}(0) = [10\ 10\ 0\ \ldots\ 0]^{\top}$.

The probability of constraint satisfaction required is $90\%$ ($\Delta = 0.1$).
Figure~\ref{fig:Quad_Mean} shows the computed trajectories by our approach and the particle based approach. Both the moment based approaches failed to compute a controller due to numerical issues. 
While the trajectories look similar for both our approach and the particle based approach, Table~\ref{tb:quadResults} shows that our approach meets the desired constraint satisfaction (0.92) via Monte Carlo but the particle based control does not (0.767) even though it determined that constraint satisfaction of its controller is 1.
The constraint violations can be seen in Figure~\ref{fig:Quad_MC}, which shows a fraction of the Monte Carlo trajectories on bottom of the red constraint set. 
In addition, while the cost at each time-step (stage cost) of each approach is similar (Figure~\ref{fig:Quad_StageCI}), the particle based approach utilizes some net thrust $u_1$ whereas our approach uses none.

\begin{table}
    \centering
    \caption{Quadcopter example: Cost and constraint satisfaction ($1-\Delta$) for both computed (Comp) and Monte-Carlo (MC) simulation for $10^5$ samples of the disturbance trajectory for our approach and the particle based approach. The open-loop and affine moment based approaches did not compute an optimal controller.}
    \label{tb:quadResults}
    \begin{tabular}{|l||c|c|c|c|c|c|}
    \hline
    \multirow{2}{*}{Method}  & \multicolumn{2}{c|}{Cost ($\times 10^3$)} & \multicolumn{2}{c|}{$1-\Delta$} & \multirow{2}{*}{Time (s)} \\\cline{2-5}
    & Comp & MC & Comp & MC & \\\hline\hline
    Chance - Open & 84.79 & 84.11 & 0.90 & 0.92 & 15.25\\ \hline
    Particle~\cite{blackmore2011chance} & 77.60 & 77.36 & 1.00 & 0.767& 237.85\\ \hline
    \end{tabular}
\end{table}

\begin{figure}[h!]
    \centering
    \includegraphics[width=\linewidth]{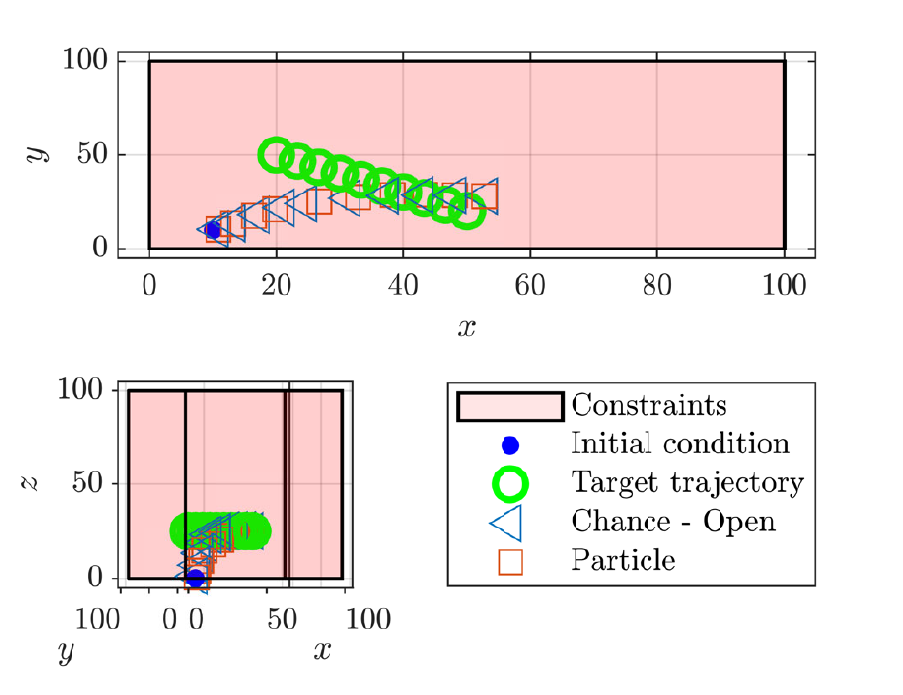}
    \caption{Mean trajectories for our approach and the particle based approach. Only our approach computes a controller that meets the constraint satisfaction when evaluated with sample trajectories (Table~\ref{tb:quadResults}).}
    \label{fig:Quad_Mean}
\end{figure}

\begin{figure}[h!]
    \centering
    \includegraphics[width=\linewidth]{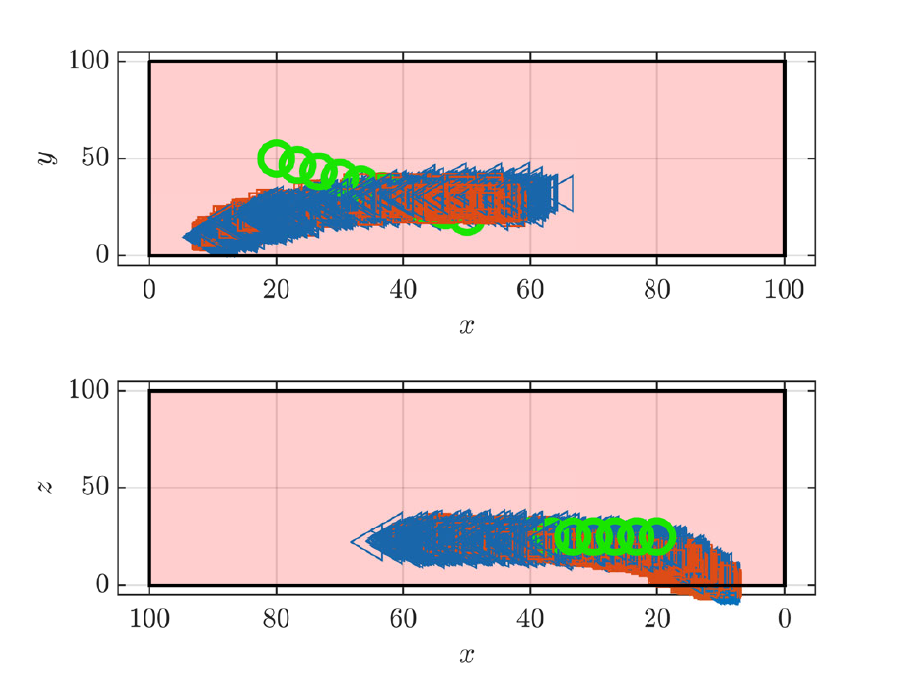}
    \caption{Monte Carlo trajectories of our approach (blue) and
    the particle based approach (orange). The constraint violation is
apparent at the bottom of the shaded region. As seen from Table~\ref{tb:quadResults}, the particle based approach has more constraint violations compared to our approach. Both moment based approaches failed to find a controller.} 
    \label{fig:Quad_MC}
\end{figure}

\begin{figure}[h!]
    \centering
    \includegraphics[width=\linewidth]{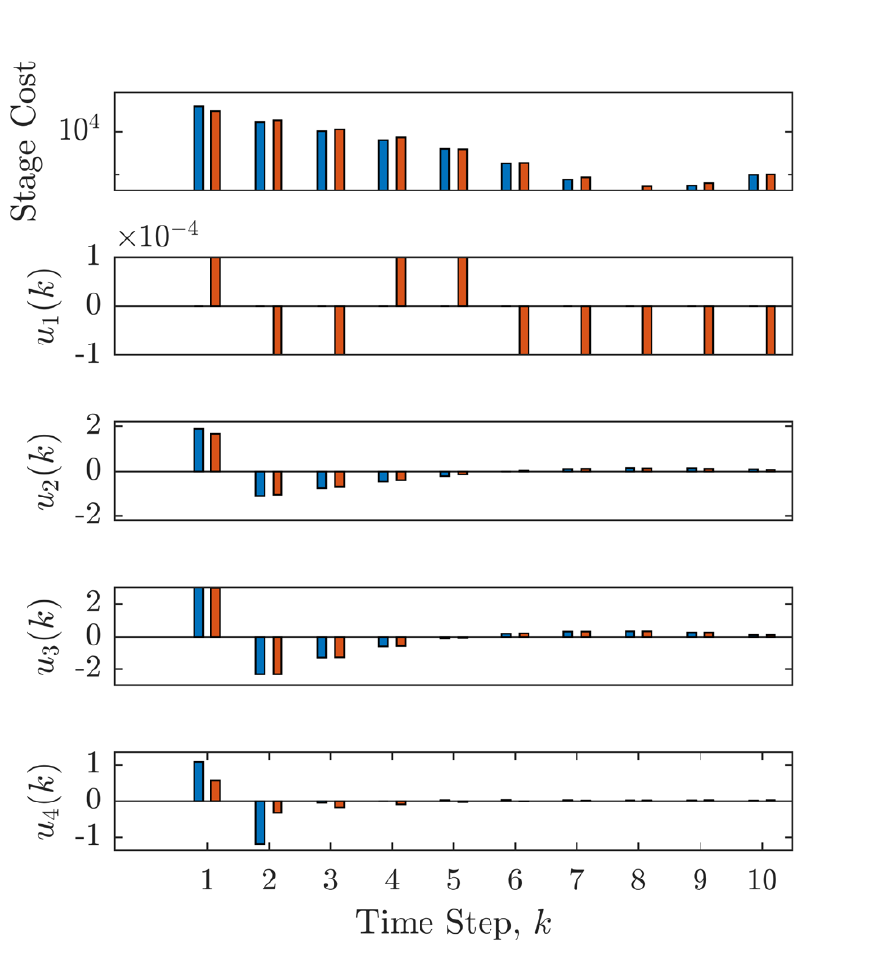}
    \caption{The cost at each time step, i.e. stage cost, and the input at each time step, for our approach and the particle based approach. Both our approach (blue) and particle based approach (orange) have similar values for $u_2, u_3$ but the particle based approach uses $u_1$ with a lower constraint satisfaction than our approach (Table~\ref{tb:quadResults}).}
    \label{fig:Quad_StageCI}
\end{figure}

\section{Conclusion}
\label{sec:conclusion}

We presented a convex optimization based approach for the
constrained, optimal
control of a linear dynamical system with additive, non-Gaussian disturbance. 
Our formulation utilizes a novel Fourier
transformation based risk allocation technique to assure
probabilistic safety for a non-Gaussian disturbance. Our
approach solves a tractable difference-of-convex program to
synthesize the desired controller. 
We make our problem amenable to standard conic solvers via
the use of piecewise affine approximations. 
Numerical experiments show the efficacy of our approach over
existing state of the art approaches, particle control and
moment based approaches, in handling non-Gaussian disturbances. 

\appendix 

\begin{figure*}
    \centering
    \includegraphics[width=0.9\textwidth]{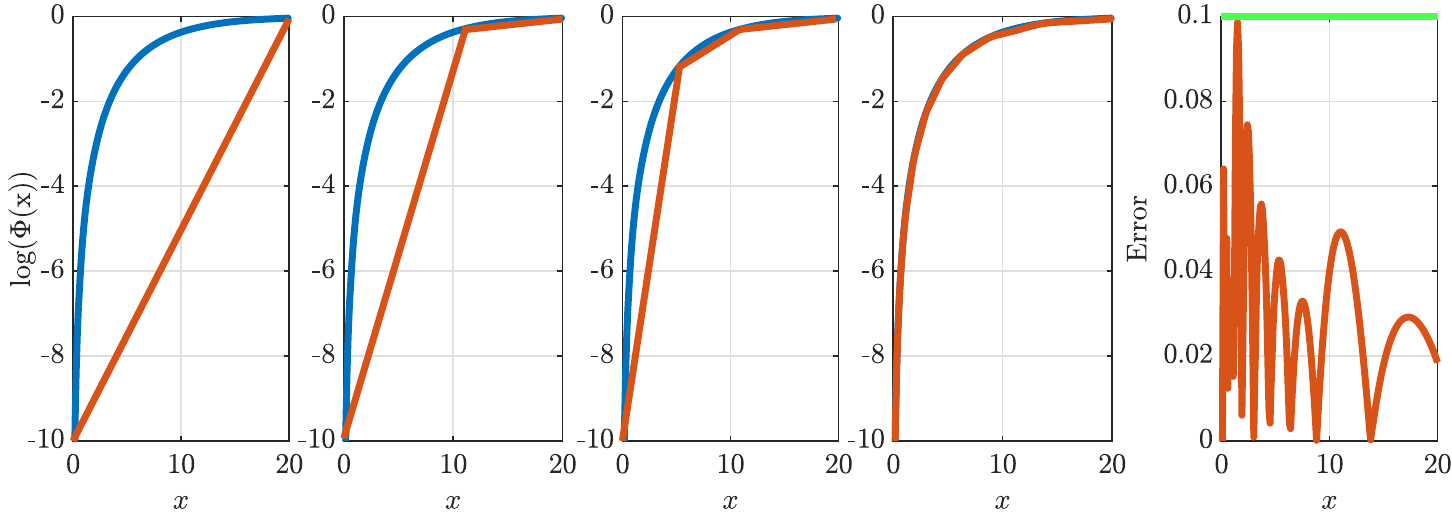}
    \caption{We compute a piecewise affine
        underapproximation (orange) of the log of an exponential
        cumulative distribution function (blue) using the sandwich
        algorithm (Appendix~\ref{ap:pwa}). The rightmost
        plot shows that the computed
    piecewise affine underapproximation keeps the error (orange) below the desired
approximation error of $\eta=0.1$ (green).}
    \label{fig:pwa}
\end{figure*}

\subsection{Difference of convex programming}
\label{ap:dc} 


We now briefly review the convex-concave procedure used to
solve difference-of-convex program \eqref{prob:dc_ex}.
Difference-of-convex programs can be solved to global optimality via general branch-and-bound methods~\cite{horst2000introduction}.
However, these methods typically require additional computational effort.
The \emph{penalty based convex-concave procedure}
(Algorithm~\ref{algo:PCCP}) is a successive convexification based method to find local optima of \eqref{prob:dc_ex} using convex optimization~\cite[Alg. 3.1]{lipp_variations_2016}.
Algorithm~\ref{algo:PCCP} relies on the observation that replacing $g_i$ with their first order Taylor series approximations in \eqref{prob:dc_ex} yields a convex subproblem, which can then be solved iteratively.
To accommodate a potentially infeasible starting point, we
relax the DC constraints using slack variables $
\overline{s}^{(k)}={[s_1^{(k)}\ s_2^{(k)}\ \ldots\ s_L^{(k)}]}^\top\in \mathbb{R}^L$, and penalize the value of the slack variables for each iteration $k$.
A possible exit condition, apart from $\tau>\tau_\mathrm{max}$, is
\begin{subequations}
\begin{align}
    \Big\vert(f_0( \overline{z}_k) -& g_0( \overline{z}_k))- (f_0(
    \overline{z}_{k+1}) - g_0( \overline{z}_{k+1})) \nonumber \\
                                    &+ \tau_{k}
                                    \sum\limits_{i=1}^L
                                    (s_i^{k} - s_i^{k+1})\Big\vert\leq \epsilon_\mathrm{dc}\label{eq:exit_cond_dc}\\
    &\hspace{4em}\sum\nolimits_{i=1}^L s_i^{k+1} \leq\epsilon_\mathrm{viol}\approx 0\label{eq:exit_cond_slack}
\end{align} \label{eq:exit_cond}%
\end{subequations}
where $\epsilon_\mathrm{dc}>0$ and $\epsilon_\mathrm{viol}>0$ are (small) user-specified tolerances.
Here, \eqref{eq:exit_cond_dc} checks if the algorithm has
converged (in the value of the objective), and \eqref{eq:exit_cond_slack} checks if $ \overline{z}_{k+1}$ is feasible.
See~\cite{lipp_variations_2016} for more details, such as convergence guarantees of Algorithm~\ref{algo:PCCP}.

\begin{algorithm}
    \caption{Local optimization of \eqref{prob:dc_ex}~\cite[Alg. 3.1]{lipp_variations_2016}}\label{algo:PCCP}
    \begin{algorithmic}[1]    
        \Require{Initial point $ \overline{z}_0$, $\tau_0 > 0$, $\tau_\mathrm{max}$, $\gamma>1$}
        \Ensure{Local optima of \eqref{prob:dc_ex}}     
        \State $k\gets 0$
        \State\algorithmicdo{}
        \State\hspace{1em} \hbox{\small $\hat{g}_i(\overline{z};\overline{z}_k)\gets g_i(\overline{z}_k) + {{\nabla g_i(\overline{z}_k)}^\top ( \overline{z} - \overline{z}_k)},\forall i\in \mathbb{N}_{[1,L]}$}
        \State\hspace{1em} Solve the following convex problem for
        $\overline{z}_{k+1}, \overline{s}^{(k)}$:\label{step:dc_prob}
        \vspace*{-0.5em}
        \begin{align}
            \begin{array}{rl}
                {\mathrm{minimize}}& f_0(
                \overline{z}_{k+1}) -
                \hat{g}_0(\overline{z}_{k+1};\overline{z}_k)
                + \tau_k  \sum\nolimits_{i=1}^L s_i^{(k)} \\
                \mathrm{subject\ to}& \overline{s}^{(k)}\succeq 0\\
       \forall i\in \mathbb{N}_{[1,L]},& f_i(
       \overline{z}_{k+1}) -
       \hat{g}_i(\overline{z}_{k+1};\overline{z}_k)\leq s_i^{(k)}\\
            \end{array} \nonumber
        \end{align}
        \State\hspace{1em} Update $\tau_{k+1}\gets \min(\gamma\tau_k,\tau_\mathrm{max})$ and $k\gets k+1$
        \State\algorithmicwhile{ $\tau\leq\tau_\mathrm{max}$ and \eqref{eq:exit_cond} is not satisfied}
  \end{algorithmic}
\end{algorithm}

\subsection{Piecewise affine underapproximations for concave functions}\label{ap:pwa}

\begin{algorithm}
    \caption{Piecewise affine underapproximations for
    a concave, differentiable function $f$}\label{algo:approx}
    \begin{algorithmic}[1]    
        \Require{Concave function $f: \mathcal{D} \rightarrow
        \mathcal{R}$, derivative $\nabla f(x)$, interval $\mathcal{D}=[x_\mathrm{min},x_\mathrm{max}]\subseteq
\mathbb{R}$, maximum underapproximation error $\eta>0$}
        \Ensure{Piecewise affine underapproximation $\ell_f^-$}     
        \State Define $\mathcal{I}$ and $\mathcal{F}$ as
        empty stacks
        \State Compute slope $m$ and intercept $c$ for the line joining $(x_\mathrm{min}, f(x_\mathrm{min}))$ and $(x_\mathrm{max}, f(x_\mathrm{max}))$
        \State Push the tuple
        $(x_\mathrm{min},x_\mathrm{max}, m, c)$ into
        $\mathcal{I}$ 
        \While{the stack $\mathcal{I}$ is not empty}
        \State Pop a tuple $(l,u,m,c)$ from $\mathcal{I}$ 
        \State Find the break point $x_{m}\in[l,u]$ s.t. $\nabla f(x_{m})=m$
        \State Set $\texttt{MaxApproxErr}$ as $f(x_{m})- (m x_{m} + c)$
        \If{\texttt{MaxApproxErr}$>\eta$} \Comment{Split $[l,u]$ at $x_m$}
            \State Compute slope and intercept $(m, c)$ for the line joining $(l,f(l))$ and $(x_m,f(x_m))$
            \State Push the tuple $(l,x_m, m, c)$ into
            $\mathcal{I}$ 
            \State Compute slope and intercept $(m, c)$ for the line joining $(x_m,f(x_m))$ and $(u,f(u))$
            \State Push the tuple $(x_m, u, m, c)$ into
            $\mathcal{I}$
        \Else\Comment{$f(x)\approx mx + c$ for $x\in[l,u]$}
            \State Push the tuple $(m,c)$ into $\mathcal{F}$
        \EndIf
        \EndWhile
        \State \Return $\ell_f^-(x)= \underset{(m,c)\in
        \mathcal{F}}{\min}(m x+c)$
  \end{algorithmic}
\end{algorithm}

Let $f: \mathcal{D} \rightarrow \mathcal{R}$ be a concave,
differentiable function defined for bounded, closed, convex,
intervals $ \mathcal{D}, \mathcal{R} \subset \mathbb{R}$.
Given a user specified approximation error $\eta > 0$,
we seek a piecewise affine underapproximation $\ell^-_{f}$ \eqref{eq:ell_defn_f} which satisfies
\eqref{eq:ell_defn_approx},
\begin{equation}
    \ell_f^-(x)\leq f(x) \leq \ell_f^-(x) + \eta. \nonumber
\end{equation}
We use $\nabla f: \mathcal{D} \to \mathbb{R}$ to denote the
derivative of $f$.

The sandwich algorithm (Algorithm~\ref{algo:approx})
constructs such an underapproximation via bisection,
specifically 
the \emph{slope-bisection rule}~\cite{rote1992convergence}.  The \emph{slope-bisection
rule} bisects a given interval $[l, u]$ at the point
$x_{m}$ such that $\nabla f(x_{m})=m =
\frac{f(u)-f(l)}{u-l}$. Due to the concavity of $f$, the
maximum error of underapproximating $f$ using a line $y=m
x + c$ with $c=f(l) - m l$ over the interval $[l,u]$ occurs
at $x_{m}$.

Algorithm~\ref{algo:approx} uses two \emph{stacks},
which are last-in first-out data structures~\cite{cormen2009introduction}. Recall that
stacks have two operations: \emph{push} to add an element to
the top of the stack, and \emph{pop} to retrieve (and
delete) the element from the top of the stack. Here, we use
the stack $\mathcal{I}$ to store the tuples associated with
intervals that must be processed to satisfy the
user-specified maximum underapproximation error $\eta$, and
the stack $ \mathcal{F}$ to store the resulting slope and intercept pairs that together define $\ell_f^-$.

To illustrate the use of Algorithm~\ref{algo:approx}, we
compute a piecewise affine underapproximation of the log of
the cumulative distribution of a non-Gaussian random
variable $\bv$. Such piecewise affine underapproximations
admit conservative enforcement of the chance constraints, as
seen in \eqref{prob:stoc_dc}. Figure~\ref{fig:pwa} shows the
approximations for the affine transformation of an exponential disturbance $a^{\top}\bw_t$ where $\bw_t = [\bw_1\ \bw_2\ \bw_3]^{\top}\in\mathbb{R}^3$ and $a = [1\ 0.5\ 0.75]^{\top}$ where the scale parameters are $\overline{\lambda}_{\bw}(k) = {[0.5\ 0.25\ 0.1667]}^\top$. 
Note that the derivative of the cumulative distribution function $\nabla \log(\Phi_{\bw}(x)) = \frac{1}{\Phi_{\bw}(x)}\psi_{\bw}(x)$ where $\psi_{\bw}(x)$ is the probability density function.
Both the cumulative distribution function and the probability density function can be evaluated from the characteristic function via Fourier inversion \cite{witov}.

\bibliographystyle{IEEEtran}
\bibliography{IEEEabrv,shortIEEE,refs}

\begin{IEEEbiography}[{\includegraphics[width=1in,height=1.25in,clip,keepaspectratio]{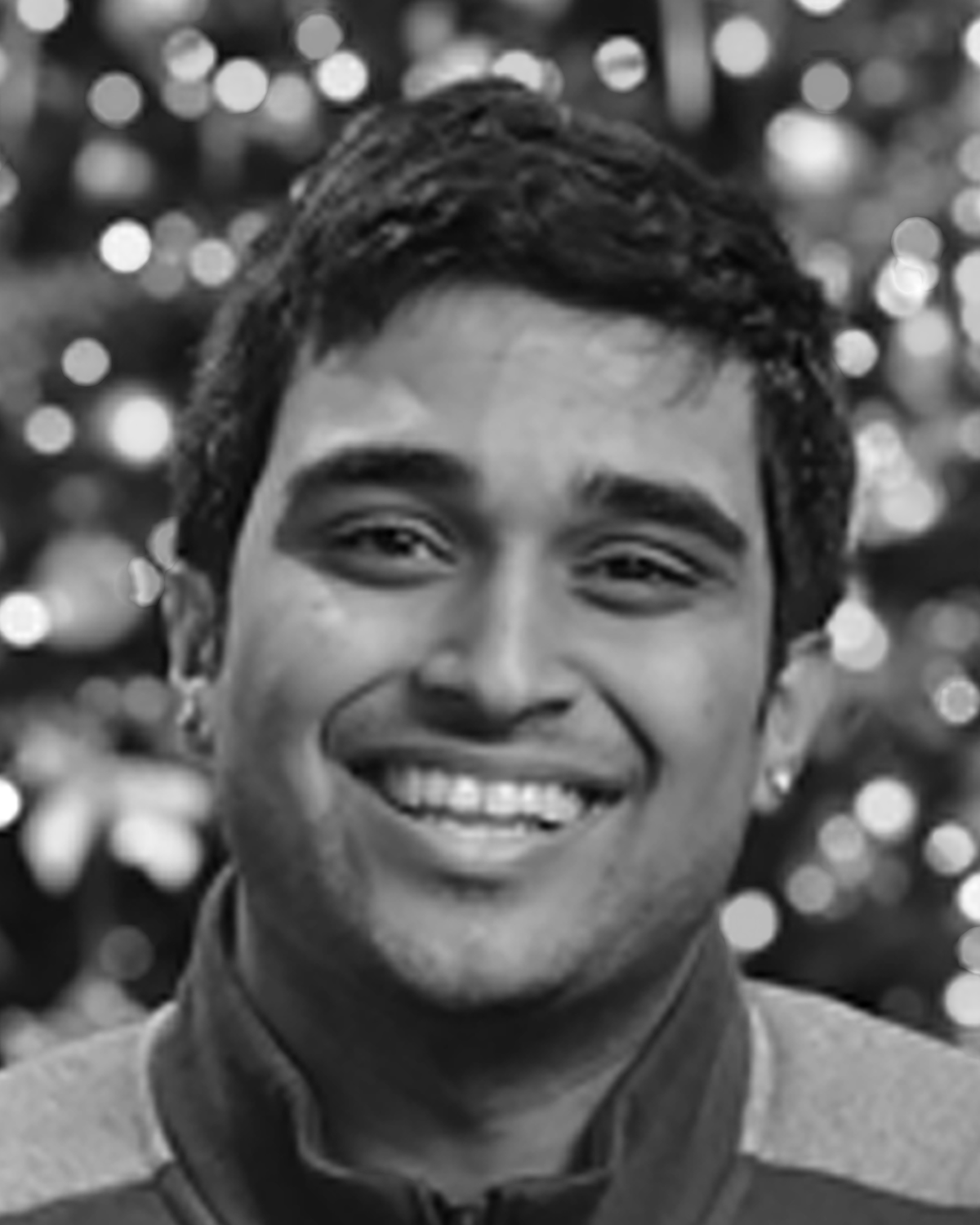}}]{Vignesh Sivaramakrishnan}(S'18) received the B.S. degree in Mechanical Engineering from the University of Utah in 2017.

He is currently pursuing a Ph.D. degree in Electrical and Computer Engineering at the University of New Mexico, Albuquerque, NM. His interest is in mathematical optimization with application to optimal control, reachability, and identifying nonlinear/stochastic systems. 
\end{IEEEbiography}
\begin{IEEEbiography}[{\includegraphics[width=1in,height=1.25in,clip,keepaspectratio]{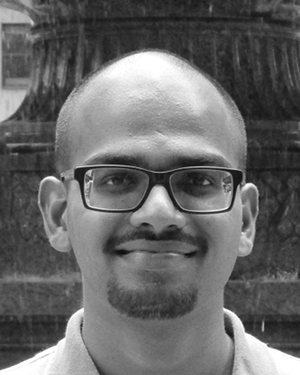}}]{Abraham Vinod}(S'15) received the B.Tech. and the M.Tech degree in
    Electrical Engineering from the Indian Institute of Technology, Madras
    (IITM), Chennai, TN, India in 2014, and a Ph.D. degree in Electrical
    Engineering from the University of New Mexico,
    Albuquerque, NM, USA in 2018. His
    research interests are in the areas of optimization,
    stochastic control, and learning.
    Dr. Vinod was awarded the Best Student Paper Award in the 2017 ACM Hybrid
    Systems: Computation and Control Conference, the
    finalist for the Best Paper Award in the 2018 ACM Hybrid
    Systems: Computation and Control Conference, the
    Prof. Achim Bopp Prize (IITM), and the Central Board of
    Secondary Education Merit Scholarship.
\end{IEEEbiography}
\begin{IEEEbiography}[{\includegraphics[width=1in,height=1.25in,clip,keepaspectratio]{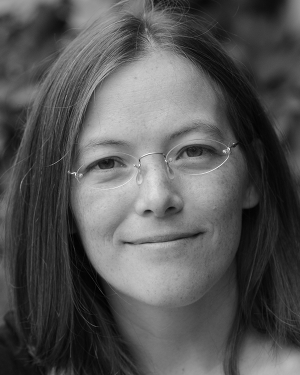}}]{Meeko Oishi}
(M'04) received the B.S.E. degree in mechanical engineering from
    Princeton University, Princeton, NJ, USA, in 1998, and the M.S. and Ph.D.
    degrees in mechanical engineering from Stanford University, Stanford, CA,
    USA, in 2000 and 2004, respectively, the Ph.D. (minor) degree in electrical
    engineering.  
    
    She is a Professor of Electrical and Computer
    Engineering with University of New Mexico, Albuquerque, NM, USA. Her
    research interests include hybrid dynamical systems, control of
    human-in-the-loop systems, reachability analysis, and motor
    control in Parkinson's disease. She previously held a
    faculty position with University of British Columbia at Vancouver and
    postdoctoral positions with Sandia National Laboratories and National
    Ecological Observatory Network.
    Prof. Oishi received the UNM Regents'
    Lectureship, the NSF CAREER Award, the UNM Teaching Fellowship, the Peter
    Wall Institute Early Career Scholar Award, the Truman Postdoctoral
    Fellowship in National Security Science and Engineering, and the George
    Bienkowski Memorial Prize, Princeton University. She was a Summer Faculty
    Fellow at AFRL Space Vehicles Directorate, and a Science and Technology
    Policy Fellow at The National Academies.
\end{IEEEbiography}

\end{document}